\documentclass[reqno,centertags, 12pt]{amsart}
\usepackage{amsmath,amsthm,amscd,amssymb,latexsym,verbatim}
\usepackage[foot]{amsaddr} 

\usepackage{graphicx,epsf}

-\marginparwidth \oddsidemargin -4mm \evensidemargin \oddsidemargin

\newtheorem{theorem}{Theorem}[section]

\newtheorem{proposition}[theorem]{Proposition}
\newtheorem{lemma}[theorem]{Lemma}

\theoremstyle{definition}
\newtheorem{definition}[theorem]{Definition}

\theoremstyle{remark}

\newcounter{smalllist}

\DeclareMathOperator*{\dist}{dist}

\DeclareMathOperator*{\sgn}{sgn}
\DeclareMathOperator*{\curl}{curl}

\allowdisplaybreaks
\numberwithin{equation}{section}

\newcommand{\lbl}{\label}

\newcommand{\ess}{\text{\rm{ess}}}

\newcommand{\beq}{\begin{equation}}
\newcommand{\eeq}{\end{equation}}

\newcommand{\bal}{\begin{align}}
\newcommand{\eal}{\end{align}}
\newcommand{\bals}{\begin{align*}}
\newcommand{\eals}{\end{align*}}

\newcommand{\bbR}{{\mathbb{R}}}

\newcommand{\bbP}{{\mathbb{P}}}
\newcommand{\bbE}{{\mathbb{E}}}

\newcommand{\calB}{{\mathcal B}}

\newcommand{\calF}{{\mathcal F}}

\newcommand{\calH}{{\mathcal H}}

\newcommand{\eps}{\varepsilon}

\newcommand{\al}{\alpha}
\newcommand{\ga}{\gamma}

\newcommand{\til}{\tilde}

\newcommand{\grad} {\nabla}

\newcommand{\dx} {\; \mathrm{d} x}
\newcommand{\dd} {\; \mathrm{d}}
\DeclareMathOperator*{\osc}{osc}
\DeclareMathOperator{\dv}{div}

\usepackage{amsfonts,latexsym,amssymb}
\newcommand{\beqn}{\begin{eqnarray}}
 \newcommand{\eeqn}{\end{eqnarray}}
 \newcommand{\be}{\begin{equation}}
 \newcommand{\ee}{\end{equation}}
 \newcommand{\ba}{\begin{array}}
 \newcommand{\ea}{\end{array}}
 \newcommand{\pa}{\partial}
 
 \newcommand{\ci}{\cite}
 \newcommand{\la}{\label}

\newcommand{\Om}{\Omega}
\newcommand{\lb}{\lambda}

\newcommand{\na}{\nabla}

\newcommand{\M}{\mathcal{M}}

\def\R{{\rm I\kern-.1567em R}}
\def\M{{\rm I\kern-.1567em M}}

\def\div {{\rm div}}

\def\dist{{\rm dist}}
\def\ess{{\rm ess}}

\def\osc {{\rm osc}}

\def\ione {par1'}

\begin{document}
\title{On divergence-free drifts}
\author{Gregory~Seregin\lowercase{$^a$}}
\address{$^a$University of Oxford, 24-29 St Giles', Oxford OX1 3LB, UK.}
\author{Luis~Silvestre\lowercase{$^b$}}
\address{$^b$University of Chicago, 5374 University Ave., Chicago, IL 60637, USA.}
\author{Vladim\' ir~\v Sver\'ak\lowercase{$^c$}}
\address{$^c$University of Minnesota, 206 Church St. SE, Minneapolis, MN 55455, USA.}
\author{Andrej~Zlato\v s\lowercase{$^d$}}
\address{$^d$University of Wisconsin---Madison, 480 Lincoln Dr., Madison, WI 53706, USA.}

\maketitle


\section{Introduction}
This paper is motivated by questions about the behavior of  solutions of
elliptic and parabolic equations with low regularity drift terms. A classical
example is
\beq \lbl{par1}
\partial_tu+b\cdot\nabla u-\Delta u = 0
\eeq
considered in $\bbR^n\times ]0,\infty[$, where $b$ is a time-dependent
vector field in $\bbR^n$.
Of particular interest to us will be the case of divergence-free $b$ (i.e.,   $\dv b = 0$), which is relevant for applications to incompressible flows.

To describe the regularity conditions on the drift term, it is useful to
recall some elementary dimensional analysis.
Equation (\ref{par1}) is invariant under the following scaling transformations:
\begin{align}
u(x,t)  & \to  u^{(\lambda)} (x,t) = u(\lambda x, \lambda^2 t)\,\, , \lbl{scaling_u} \\
b(x,t)  & \to b^{(\lambda)} (x,t) = \lambda b(\lambda x,\lambda^2 t)\,\, ,\lbl{scaling_b}
\end{align}
where $\lambda >0$. Following the usual convention (see, e.g., \cite{CKN}),
we can say that $u$ has dimension $0$ and $b$ has dimension $-1$.
The classical theory (see, e.g., \cite{LSU}) studies the question of under which
conditions \eqref{par1}
can be considered as a perturbation
of the heat equation. The required regularity on $b$ is usually
expressed as $b\in \calB$, with $\calB$ a suitable function space.
Typically the borderline spaces for which one can still prove most of the
deeper results\footnote{such as, for instance, the Harnack inequality for positive solutions}
are scale-invariant under the scaling (\ref{scaling_b}) of $b$, that is,  $||b^{(\lambda)}||_{\calB}=||b||_{\calB}$  (see, e.g., \cite{LSU, Nazarov-Uraltseva})~\footnote{
For example, the Lebesgue spaces $L_{q,p}=L^p_t L^q_x$
are scale-invariant if and only if $2/p+q/n=1$.}.
The reason for this is as follows. The arguments in the proofs of the
``deeper properties"\footnote{The definition of what is meant by a ``deeper property" is of course somewhat
ambiguous. We already mentioned the Harnack inequality as an example. On the other hand, the weak maximum principle would not be considered as such in this context.}
typically have to work on all (small) scales
and  we therefore need to control $b$ on all scales, which naturally leads to the
scale-invariant spaces.

Similar considerations can be made for elliptic equations of the form
\beq \lbl{ell1}
-\Delta u + b\cdot\nabla u = 0\,\,,
\eeq
with the elliptic scaling
\begin{align}
u(x)  & \to  u^{(\lambda)} (x) = u(\lambda x)\,\, , \lbl{e_scaling_u} \\
b(x)  & \to b^{(\lambda)} (x) = \lambda b(\lambda x)\, .\lbl{e_scaling_b}
\end{align}

Let us now consider the condition $\dv b=0$ and its consequences. (The relevant
references include, for example, \cite{Zh, Osada} in the parabolic case and
\cite{Mazja, Kontovourkis} in the elliptic case.)
Among the most important consequences 
are the following\footnote{To derive these consequences, one needs
to assume that the formal integration by parts used to obtain
them is valid. We are ignoring this technical issue for the moment.}.

\noindent
(i) The energy identity
\beq \lbl{energy}
\int_{\bbR^n} |u(x,t_2)|^2\, dx + \int_{t_1}^{t_2}\int_{\bbR^n}|\nabla u|^2\, dx\,dt =
\int_{\bbR^n} |u(x,t_1)|^2\, dx
\eeq
is exactly the same as for the heat equation.

\noindent
(ii) The integral $\int_{\bbR^n} u(x,t)\,dx$ is conserved:
\beq \lbl{l1}
\int_{\bbR^n}u(x,t_2)\, dx = \int_{\bbR^n}u(x,t_1)\, dx\, .
\eeq

J. Nash showed in his famous paper \cite{N}  (inequality (8) on page 936)
that one can obtain from (i) and (ii) the point-wise upper bound
\beq \lbl{nash}
|G(x,t;y,s)|\le \frac C {(t-s)^{n/2}}\,
\eeq
on the fundamental solutions $G(x,t;y,s)$, with $C$ depending only on the dimension.
Therefore, this bound also holds for solutions of (\ref{par1}) when $\dv b=0$,
with practically no other assumptions on $b$. The heuristic behind this
estimate is that in an incompressible fluid, mixing can  enhance the decay
of, say, a temperature field but it cannot slow it down. Nash's simple argument
proving this heuristics is very elegant. There are many other results in this direction,
see for example \cite{Zh, Osada}.
Bound~(\ref{nash}) can also be integrated in time to obtain (global) estimates of
$\sup_x |u(x)|$ for the elliptic
problem
\beq
- \Delta u + b \cdot\nabla u = f\, ,
\eeq
 with $f \in L_{n/2+\delta}$,  a divergence-free $b$, and practically no other assumptions.

Since the condition $\dv b = 0$ has such strong consequences for the
$L_\infty$-bounds, it is natural to ask about its effects on
other properties of the solutions. For instance, can the standard assumptions
on the drift term $b$ needed, say, for the Harnack inequality
be relaxed when $\dv b = 0$? Similar questions have been
considered, for example, in \cite{Osada, Zh}.

It turns out that the condition $\dv b =0$ can be used to relax
the regularity assumptions on $b$ under which one can
prove the Harnack inequality and other results. However,
the effects are not as dramatic as in the case of Nash's
upper bound~(\ref{nash}), for which
not much is needed beyond $\dv b = 0$.
In particular, it seems that even with the condition $\dv b=0$
one cannot significantly ``break the scaling". Indeed, to be able to prove the
``deeper regularity properties" of the solutions
(as discussed above), we still need to assume
that $b$ belongs, at least locally, to a scale-invariant
space~$\calB$. The norm can be weaker than in the absence
of the assumption $\dv b = 0$, but it still has to
be scale invariant or stronger on the small scales.
For example, the results of \cite{Osada} imply that
the Harnack inequality, H\"older continuity of solutions,
and the Aronson estimate
for fundamental solutions\footnote{
$c_1(t-s)^{-n/2}\exp[{c_2|x-y|^2/(t-s)}]\le G(x,t;y,s)\le
c_3(t-s)^{-n/2}\exp[{c_4|x-y|^2/(t-s)}]$
, see \cite{Ar}.
}
remain true when $b\in L_{\infty} (L_{\infty}^{-1})$,
where $L_{\infty}^{-1}$ denotes distributions which are
first derivatives of bounded measurable functions.
This should be compared to the condition $b\in L_{n,\infty}$,
which naturally comes up when the assumption $\dv b = 0$ is
dropped\footnote{Strictly speaking, as far as we are aware,
when we do not assume $\dv b=0$,
most of the regularity results above are proved for
$b\in L_{q,p}$ with $2/p+n/q=1$ and $p<\infty$
(see \cite{Nazarov-Uraltseva}), but not in the borderline
case $p=\infty, \,\,q=n$.}.
Note that both $L_n$ and $L_\infty^{-1}$ are scale-invariant.

The assumption $\dv b=0$ can be used to reformulate equation~(\ref{par1})
in the following way. When $\dv b=0$, we can write
$b=\dv d$ for an anti-symmetric tensor
$d=(d_{ij})$\footnote{For $n=3$, this corresponds to introducing the vector
potential $\tilde d$ such that  $b=\curl \tilde d$.}.
Moreover, by introducing a suitable
``gauge condition"\footnote{such as $ d_{kl,j}+ d_{jk,l}+ d_{lj,k}=0$,
which for $n=3$ and $b=\curl \tilde d$ corresponds to $\dv \tilde d = 0$},
we can assume that the derivatives of $d$ have similar regularity as $b$.
Since $b$ has dimension $-1$ with respect to the natural scaling of~(\ref{par1}),
the tensor $d$ has dimension 0, that is, it scales as
\beq \lbl{scaling_c}
d(x,t) \to d(\lambda x, \lambda^2 t)
\eeq
when $u$ is scaled by~(\ref{scaling_u}).

Replacing $b$ by the potential $d$, equation~(\ref{par1}) becomes
\beq \lbl{par1'}
\partial_t u-{\rm div} (A\nabla u) = 0
\eeq
where $A=\mathbb I+d$.
This is a divergence-form equation with a non-symmetric leading term.
Such equations (including the versions with lower-order terms) have been
studied in \cite{Osada} under the assumptions that the coefficients
$a_{ij}$ are bounded measurable functions satisfying the ellipticity
condition
\beq\lbl{ellipticity}
(A\xi)\cdot\xi\ge\nu|\xi|^2\,.
\eeq
The results of \cite{Osada} show, roughly speaking, that most of the results which are
valid for symmetric $A$ are also true in the non-symmetric case.
The transformation of~(\ref{par1}) to~(\ref{par1'}) has been used in many other
works (see, for instance, \cite{Papanicolaou}).

In the elliptic case, Mazja and Verbitsky \cite{Mazja} studied  (among other things) the bi-linear
form
\beq \lbl{bilinear}
(u,v) \ \to \int_{\bbR^n} (A \nabla u)\cdot \nabla v \, dx\, .
\eeq
The form is obviously continuous in $\dot H^1$ when $A$ is bounded,
but it turns out that the boundedness of the coefficients is not a necessary
condition for the boundedness of the form. The form is still continuous
on $\dot H^1$ if the symmetric part of $A$ is bounded and the anti-symmetric
part of $A$ is in the John-Nirenberg space $BMO$ (bounded mean oscillation).
This is a consequence of the following two facts:

\noindent
(i) If $A$ is anti-symmetric,
the form~(\ref{bilinear}) can be factored through the determinants ${\frac {\partial(u,v)}{\partial(x_i,x_j)}}$.

\noindent
(ii) The determinants have ``better than expected" regularity: when $u,v$ are in $\dot H^1$, the
determinants are not only in $L_1$, but they are in fact in the Hardy space $\calH^1$, the dual space of $BMO$ (see \cite{CMLS}).

It is natural to expect that much of the classical regularity results for
elliptic and parabolic equations with measurable coefficients in divergence
form will remain valid if the leading part $A$ is of the
form $A=a+d$, with $a$ symmetric, bounded
and satisfying the usual ellipticity condition~(\ref{ellipticity}),
and $d$ anti-symmetric and belonging to $BMO$ in the elliptic
case, and to $L_\infty(BMO)$ in the parabolic case.

Indeed, let $Q_-=\mathbb R^n\times \mathbb R_-$ (with $\mathbb R_-=]-\infty,0[$) and assume that
\be\la{i2}A=a+d,\ee
where  $a\in L_\infty(Q_-;\mathbb M^{n\times n})$ is a symmetric matrix satisfying
\be\la{i3}\nu \mathbb I\leq a\leq \nu^{-1}\mathbb I\ee
and $d \in L_\infty(\mathbb R_-;BMO(\mathbb R^n;\mathbb M^{n\times n}))$
 is a skew symmetric matrix, that is,
\be\la{i4}d=-d^* \ee
for all $(x,t)\in Q_-$. Here $\nu>0$, $\mathbb I$ is the identity in the space $\mathbb M^{n\times n}$ of $n\times n$-matrices and $d^*$ is the transpose of $d$. Let also $B(x,r)$ be the ball of radius $r$ centered at $x\in\mathbb R^n$, and $Q(z,r)=B(x,r)\times ]t-r^2,t[$ a parabolic ball in $\bbR^{n+1}$ centered at point $z=(x,t)$. Finally, let $B=B(0,1)$ and $Q=Q(0,1)$.
We then prove the following parabolic  Harnack inequality and Liouville theorem for  suitable weak solutions (see Definition \ref{id1}) to \eqref{par1'}.

\begin{theorem}\la{nt1}
If the matrix $A$ satisfies  conditions (\ref{i2})--(\ref{i4}), then there exists  $C>0$, depending only on $n$, $\nu$, and $\|d\|_{L_\infty(-1,0;BMO(B))}$, such that for any nonnegative suitable weak solution $u$ to (\ref{par1'}) on $Q$  we have
\be\la{n1}
\sup\limits_{(y,s)\in Q(z_R,R/2)}u(y,s)\leq C\inf\limits_{(y,s)\in Q(z,R/2)}u(y,s),
\ee
whenever   $Q(z,R)\subset Q$. Here, $z_R=(x,t-R^2/2)$.

\begin{theorem}\la{it2}
If the matrix $A$ satisfies  conditions (\ref{i2})--(\ref{i4}), then the only  bounded  ancient suitable weak solutions to (\ref{par1'}) on $Q_-$ are the constant functions.
\end{theorem}

\end{theorem}

{\it Remark.} Of course, the corresponding elliptic results follow immediately by taking time-independent solutions.  In addition, in Section 3 we provide a second --- short and elementary ---  proof of the Liouville theorem for weak (sub)solutions (see Definition \ref{definition of weak solutions}) to \eqref{ell1} in $\bbR^2$.
\medskip

Recall that the  norm in the space $BMO(\Omega;\mathbb M^{n\times n})$ is
$$\|d\|_{BMO(\Om;\mathbb M^{n\times n})}=\sup\left\{\frac 1{|B(0,r)|}\int\limits_{B(x,r)} \left| d
-[d]_{x,r} \right| dx:\,B(x,r)\Subset\Omega \right\},$$
with $[d]_{x,r}$ the average of $d$ over $B(x,r)$.

We note that the space $BMO$ is invariant under the scaling~(\ref{scaling_c}),
and hence these results are again in line with the argument that to preserve the
``deeper properties" of the solutions, one cannot ``break the scaling".
One of the goals of this paper is to present some evidence for this based
on studying the
failure of Liouville theorems under appropriate conditions.

Let us first look at  \eqref{ell1} in $\bbR^n$.
By the Liouville theorem for~(\ref{ell1}) we mean the usual statement that
a bounded solution in $\bbR^n$ has to be constant. This is of course
true for $b\equiv 0$.  For the time being let us assume that the vector
field $b$ is locally smooth, hence the solutions $u$ are also locally
smooth and the only obstacles to the validity of the Liouville theorem
are global.

The results of Stampacchia \cite{Sta} imply the following:

\medskip
\noindent
(L) {\sl If $b\in L_n(\bbR^n)$, then the Liouville theorem for~(\ref{ell1})
holds.}
\medskip

\noindent
This is easy for $n=1$, and for $n=2$, there is also a relatively simple proof
based on the energy estimate. The proof for $n\ge 3$ can be accomplished
by using the H\"older estimate or the Harnack inequality (see Sections 7 and 8 of
\cite{Sta}). If $n\ge 3$, then by Theorem 2.3 in \cite{Nazarov-Uraltseva},
 the condition on $b$ can be weakened
to $\lim \inf_{R\to\infty}\sup_{|x|=R} ||b||_{L_n(B(x,R\delta))}<c_n$ for some $\delta>0$,
where $c_n>0$ is a fixed dimension-dependent constant.
 This result implies in particular that (L) remains true when
 \beq\lbl{b_decay}
 |b(x)|\le {\frac C {|x|}}\quad \quad \text{for large $|x|$.}
 \eeq
 It is not clear to us whether condition~(\ref{b_decay}) is sufficient
 in dimension $n=2$, but this should not be a hard question.

 In dimension $n=1$,  condition~(\ref{b_decay}) is sufficient
 when $C\le 1$, as one can check by direct integration. With $C>1$, however, \eqref{b_decay}
is no longer sufficient. This can be seen from
 the  example
 \beq\lbl{example}
 b(x)  = {\frac {2x}{1+x^2}} \qquad \text{and} \qquad  u(x)  =\arctan(x)\,,
 \eeq
 which was  pointed out in this context to one
 of the authors in 1997 by Joel Spruck.
 The trivial extension of this example to higher dimensions is
 \beq\lbl{example_n}
 b(x_1,\dots,x_n)  = \left({\frac {2x_1}{1+x_1^2}},0,\dots,0\right) \qquad \text{and} \qquad  u(x_1,\dots,x_n)  =\arctan(x_1)\, .
 \eeq
 We note that the vector field $b$ in~(\ref{example_n}) belongs
 to the space $(BMO)^{-1}(\bbR^n)$, since
 \beq
 {\frac {2x}{1+x^2}}={\frac d {dx}}\log(1+x^2)
 \eeq
 and  $\log(1+x_1^2)\in BMO(\bbR^n)$.

This example and Theorem \ref{it2}, which establishes the Liouville theorem for $b\in (BMO)^{-1}$ and $\dv b=0$,
 together show that the divergence-free
 condition can play an important role in Liouville
 theorems for equations with drift terms.
 On the other hand, we now provide a counter-example
to the Liouville theorem with a divergence-free $b$ on $\bbR^2$ which is in some sense
 not too far from $(BMO)^{-1}$. Recall that the  {\it stream function} of a divergence-free vector field $b$ on $\bbR^2$ is $H:\bbR^2\to\bbR$ such that
\beq \lbl{stream2}
b(x)=\nabla^\perp H(x) = (H_{x_2}(x), -H_{x_1}(x)).
\eeq
 We therefore have
\beq\lbl{stream3}
-\Delta  + b\cdot\nabla  = -{\rm div} ( A\, \nabla),
\eeq
where $A(x)=\mathbb I+d(x)$ has skew-symmetric part
 $$d(x)=\left( \begin{matrix} 0 & H(x) \cr -H(x) & 0 \cr \end{matrix} \right) .
$$

\begin{theorem} \lbl{T.1.2}
There exists a divergence-free vector field $b\in C^\infty(\bbR^2)$ with all derivatives bounded and a stream function satisfying $|H(x)|\le C\ln|x|\ln\ln|x|$ for some $C$ and all large enough $|x|$ such that \eqref{ell1} has a non-constant  bounded classical solution.
\end{theorem}

 This  illustrates, to some
 degree, the important role of scale invariance of the assumptions
 in the Liouville result. In particular, it seems unlikely that one
 can significantly ``break the scaling" even if we assume that $\dv b = 0$.

 We conjecture that similar negative conclusions can be arrived at when
 considering questions about H\"older continuity of  solutions
 of (\ref{ell1}) (as well as the Harnack inequality). For example, it seems unlikely that
 the condition $\dv b = 0$
 is sufficient to get a $C^\alpha$-bound on  solutions
 $u$ in the unit ball $B=B(0,1)$
 under the assumptions $|u|\le C$ and $||b||_{L_{n-\delta}}\le C$.
 (Here we assume that all the functions involved are smooth, but only
 the indicated quantities are controlled, and we are interested
 in an a-priori bound.)

Related to this are our last two main results, concerning distributional solutions $u$ (see Proposition \ref{p:h1-from-linfty}) of \eqref{ell1} in $B$ with divergence-free $b\in L_1(B)$. The first establishes a logarithmic modulus of continuity of such solutions in two dimensions, depending only on $\|b\|_{L_1(B)}$ and $\|u\|_{L_\infty(B)}$. However, due to the low regularity assumed on the vector field $b$ and $u$ solving \eqref{ell1} only in the distributional sense,  our result is restricted to those solutions which can be obtained as weak-star $L_\infty$-limits of solutions with drifts in $L_2(B)$.

\begin{theorem}\label{limit}
Let $B$ be the unit ball in $\bbR^2$ and let $(b_m,u_m)\in L_2(B)\times L_\infty(B)$ be a sequence of divergence-free drifts $b_m$ and distributional solutions $u_m$ to (\ref{ell1}) with $b=b_m$. Assume that $u_m$ are uniformly bounded in $B$ and
$$b_m \to b\qquad {\rm in } \quad L_1(B),$$
$$u_m\, { \stackrel{\star}{\rightharpoonup}}\, u\qquad {\rm in } \quad L_\infty(B).$$
Then the function $u$ is a distributional solution to (\ref{ell1}). Moreover,
$$u\in H^1_{\rm loc}(B)\cap C_{\rm loc}(B)$$
and at the origin $u$ has the modulus of continuity
\beq\lbl{modcont}
 \sup_{x \in B(0,r)} |u(x)-u(0)| \leq \frac {C\left(1+||b||_{L_1(B)}\right)^{1/2}} {\sqrt{-\log r}} ||u||_{L_\infty(B)}
 \eeq
with a universal $C>0$.
\end{theorem}

In three (and more) dimensions this result is false. Indeed, there exists no modulus of continuity of classical solutions depending only on $\|b\|_{L_1(B)}$ and $\|u\|_{L_\infty(B)}$, and distributional solutions  $u\in L_\infty(B)\cap H^1(B)$ with divergence-free $b \in L_1(B)$ may be discontinuous.

\begin{theorem} \lbl{T.3.4} Let $B$ be the unit ball in $\bbR^3$.

(i) There is $c>0$ such that for each $\eps>0$ there is a smooth divergence-free drift $b$ with $||b||_{L_1(B)}\le c$ and a smooth $u$ with $||u||_{L_\infty(B)}\le 1$, solving \eqref{ell1} in $B$ and satisfying
\[
u(0,0,\eps)-u(0,0,0)\ge c^{-1}.
\]

(ii) There is a divergence-free drift $b\in L_1(B)$   and a distributional solution $u\in H^1(B)\cap L_\infty(B)$ of \eqref{ell1} in $B$ which can be approximated by a smooth sequence $(b_m,u_m)$ in the sense of Theorem
\ref{limit}, but $u$ is discontinuous at the origin. 
\end{theorem}

Our paper is organized as follows. In the next section, we develop  local regularity theory
for parabolic operators (\ref{par1'}) under the assumption that the skew-symmetric part of $A$ is in BMO, and prove Theorems \ref{it2} and \ref{nt1}. The important step of our approach is a higher integrability of suitable weak solutions. This allows us to adopt Moser's method for proving the Harnack inequality that implies H\"older continuity
of suitable weak solutions and Liouville type theorems for ancient suitable weak solutions. All these results hold true for the heat equation with a drift $b\in L_\infty(BMO^{-1})$ as a particular case. In this connection, we would like to mention the recent paper \cite{FrVi}, of which we learned while writing the present manuscript. In  \cite{FrVi}, among other questions, the Cauchy problem for the heat operator with the drift term from  $L_\infty(BMO^{-1})$ has been considered and H\"older continuity of solutions has been proved. The authors of \cite{FrVi} follow the Caffarelli-Vaseur approach \cite{CaVa}. In Section 3, an elementary proof of a Liouville theorem in the two-dimensional elliptic case is provided and Theorem \ref{T.1.2} is proved. Theorems \ref{limit} and \ref{T.3.4} are proved in Section 4.

\medskip
{\bf Acknowledgement.} GS was partially supported by the RFFI grant
08-01-00372-a. The other authors were supported in part by NSF grants DMS-1001629 (LS),  DMS 0800908 (VS), and DMS-0901363 (AZ).  LS and AZ  also acknowledge partial support by Alfred P.~Sloan Research Fellowships.

\section{Some results for parabolic equations}
\label{s:parabolic}

The main goal of this section  is to prove Theorems \ref{it2} and \ref{nt1}.
We consider \eqref{\ione} in $Q_-=\mathbb R^n\times \mathbb R_-$, with the matrix $A$ satisfying \eqref{i2}--\eqref{i4}.
We will study the so-called suitable weak solutions to (\ref{\ione}).  In what follows we will use the  abbreviated notation
$$B(r)=B(0,r),\quad B=B(1), \quad Q(r)=Q(0,r),\quad Q=Q(1),$$
as well as $z=(x,t)$.

\begin{definition}\la{id1} Function $u$ is said to be a {\it suitable weak solution} to  equation (\ref{\ione}) in the parabolic ball $Q(R)$
if it satisfies
\be\la{i5}u\in L_{2,\infty}(Q(R))\cap W^{1,0}_2(Q(R)), \ee
\be\la{i6}\int\limits_{Q(R)}u\,\pa_t\varphi dz=\int\limits_{Q(R)} (A \na u)\cdot \na\varphi dz\qquad \forall\varphi \in C^\infty_0(Q(R)),\ee
and for a.e.~$t_0\in ]-R^2,0[$, the local energy inequality
$$\frac 12 \int\limits_{B(R)}\varphi(x,t_0)|u(x,t_0)|^2dx+\int\limits^{t_0}_{-R^2}
\int\limits_{B(R)}\varphi \na u\cdot a\na u dz\leq $$
\be\la{i7}\leq\frac 12\int\limits^{t_0}_{-R^2}\int\limits_{B(R)}|u|^2\pa_t\varphi dz-\int\limits^{t_0}_{-R^2}\int\limits_{B(R)}(A\na u)\cdot \na\varphi u dz\ee
holds for all non-negative test-functions $\varphi \in C^\infty_0(B(R)\times ]-R^2,R^2[)$.

The function $u:Q_-\to \mathbb R$ is called an {\it ancient suitable weak solution} to (\ref{\ione}), if it is a suitable weak solution to (\ref{\ione}) in $Q(R)$ for any $R>0$.
\end{definition}

It is not clear whether one can show that any solution to (\ref{\ione}), subject
to assumptions (\ref{i5}) and (\ref{i6}), satisfies  local energy inequality (\ref{i7}). In this respect the
situation is similar to the Navier-Stokes equations: there is a certain cancelation due to the skew symmetric matrix $d$ which works well in global setting, i.e., when initial-boundary value problems are under consideration. The corresponding procedure is relatively routine and  leads to the existence of global solutions which satisfy the inequalities in Definition \ref{id1} at least locally.




We now outline the main points of our approach.
The structure of  equation (\ref{\ione}) admits a modification of the technique developed by
J. Moser in \ci{M1}--\ci{M2} and  get H\"older continuity of suitable weak solutions.
This property, together with  scaling invariance, leads to the Liouville theorem.
The main tool of proving H\"older continuity is the Harnack inequality. We prove the
Harnack inequality for smooth solution by the method of J. Moser. Extension of the
Harnack inequality to suitable weak solutions is provided by higher integrability
of the the spatial gradient. Here, our arguments use an approach due to
M. Gianquinta and M. Struwe, see \ci{GS}.

\subsection{Local set-up and higher integrability }

Equation (\ref{\ione}) is invariant with respect to
translations and the following scaling
\be\la{lh1}u^\lambda(x,t)=u(\lambda x,\lambda^2t),\qquad A^\lambda(x,t)=A(\lambda x,\lambda^2t)\ee
for any positive $\lambda$. This allows us to reduce all considerations
to some canonical domain, say, to $Q=Q(1)$.

So, we consider equation (\ref{\ione})
in the unit parabolic cylinder.
Matrix $A$ is
split
into two parts as in (\ref{i2})
with matrices $a\in L_\infty(Q;\mathbb M^{n\times n})$ and
$d \in L_\infty(-1,0;BMO(B;\mathbb M^{n\times n}))$ satisfying  conditions (\ref{i3})
and (\ref{i4})



In what follows, we shall denote by $c$ positive constants depending only on $n$ and
$\nu$. We let $\|d\|_{L_\infty(BMO)} =\|d\|_{L_\infty(-1,0;BMO(B))}$
and denote mean values by
$$[f]_{x,r}=\frac 1{|B(r)|}\int\limits_{B(x,r)}f(y)dy,
\qquad (u)_{z_0,r}=\frac 1{|Q(r)|}\int\limits_{Q(z_0,r)}u(z)dz.$$

The main result of this subsection is the following theorem.
\begin{theorem}\la{t21} Assume that $u$ is a suitable weak solution to (\ref{\ione})
in $Q$ and matrices $A$, $a$, and $b$ satisfy conditions
(\ref{i2})-(\ref{i4}). Then there exist two positive constants $p>2$
and $C$  depending only on $n$, $\nu$, and $\|d\|_{L_\infty(BMO)}$
such that $u\in L_p(Q(R))$ for any $R\in ]0,1[$. Moreover, the
following estimate is valid: \be\la{21}\Big(\frac
1{|Q(R)|}\int\limits_{Q(z_0,R)}|\na {u}|^pdz\Big)^\frac 1p \leq
C\Big(\frac 1{|Q(6R)|}\int\limits_{Q(z_0,6R)}| \na
{u}|^2dz\Big)^\frac 12\ee for all $Q(z_0,6R)\subset Q$ with
$6R<\dist{(x_0,\pa B)}$ and $t_0-(6R)^2>-1$.\end{theorem}

This theorem is a consequence of the reverse H\"older inequality, see  \ci{GS} for further references.
To prove the reverse H\"older inequality,
we need a Caccioppoli's type inequality. To formulate it, let us introduce  additional notation. Fix a non-negative cut-off functions $\varphi\in C_0^\infty (B(2))$
and $\chi_0 (t)$ with the following properties:
$$\varphi(x)=1 \qquad x\in B,\qquad\chi(t)=0 \quad t\leq -4,$$$$\qquad \chi_0(t)=(t+4)/3 \quad -4<t<-1,\qquad \chi_0(t)=1\quad t\geq -1.$$ Now, for a point $z_0=(x_0,t_0)$ and for $R>0$ such that $Q(z_0,2R)\in Q$, we let
$$\chi_{t_0,2R}(t)=\chi_0((t-t_0)/R^2),\qquad \varphi_{x_0,2R}(x)=\varphi((x-x_0)/R).$$ And then we can introduce a mean value of $u$ as in \ci{GS}
$$u_{x_0,2R}(t)=\int\limits_{B(x_0,2R)}u(x,t)\varphi^2_{x_0,2R}(x)dx
\Big(\int\limits_{B(x_0,2R)}\varphi^2_{x_0,2R}(x)dx \Big)^{-1}.$$
In our particular situation, we have
\begin{lemma}\la{l22}(Caccioppoli's type inequality) Under assumptions of Theorem \ref{t21}, the following inequality is valid:
$$\frac 12\int\limits_B|\widehat{u}(x,t_0)|^2
\varphi^2_{x_0,2R}(x)dx+
\nu\int\limits_{-1}^{t_0}\int\limits_B \chi^2_{t_0,2R}\varphi^2_{x_0,2R}|\na \widehat{u}|^2dz\leq$$
\be\la{22}\leq\frac 12\int\limits_{-1}^{t_0}\int\limits_B |\widehat{u}|^2 \varphi^2_{x_0,2R}\pa_t\chi^2_{t_0,2R}dz
-\int\limits_{-1}^{t_0}\int\limits_B\chi^2_{t_0,2R}(a\na\widehat{u})
\cdot\na\varphi^2_{x_0,2R}\widehat{u}dz\ee
$$-\int\limits_{-1}^{t_0}\int\limits_B\chi^2_{t_0,2R}((d-[d]_{x_0,2R})\na\widehat{u})
\cdot\na\varphi^2_{x_0,2R}\widehat{u}dz,$$
where
$$\widehat{u}(x,t)=u(x,t)-u_{x_0,2R}(t).$$ Inequality (\ref{22}) holds for a.a. $t_0\in ]-1,0[$, for all $x_0\in B$, and for all $R>0$ subject to the additional condition $Q(z_0,R)\subset Q$.
\end{lemma}


\begin{proof} There are two important points to note. The first one is that for any skew-symmetric matrix $d_0$, depending on $t$ only, we have
\be\la{23}\int\limits_Qd_0\na u\cdot \na \varphi udz=0\ee
whenever $\varphi\in C^\infty_0(Q)$.  The proof is straightforward integration by part.

The second point is that, see \ci{GS},
\be\la{24}
\pa_t u_{x_0,2R}\in L_\frac 32(-1,0).\ee
To see this, we take as test function in (\ref{i6}) the function $\varphi^2_{x_0,2R}(x)\eta(t)$
and conclude
\be\la{25}\pa_t u_{x_0,2R}(t)=-\int\limits_{B(x_0,2R)}A(z)\na u(z)\cdot \na\varphi^2_{x_0,2R}(x)dx\Big /\int\limits_{B(x_0,2R)}\varphi^2_{x_0,2R}(x)dx\ee

Next, we replace $u(x,t)$ with $\widehat{u}(x,t)+u_{x_0,2R}(t)$ in local energy inequality (\ref{i7}) and take $\varphi=\chi^2 \varphi^2_{x_0,2R}$ with $\chi$ from $C^1_0(-1,1)$. Then  terms which do not contain spatial derivatives can be transformed as follows
$$\frac 12\int\limits_{B(x_0,2R)}|\widehat{u}(x,t_0)+u_{x_0,2R}(t_0)|^2\varphi^2_{x_0,2R}(x)
dx=$$$$=
\frac 12\int\limits_{B(x_0,2R)}|\widehat{u}(x,t_0)|^2\varphi^2_{x_0,2R}(x)dx+\frac 12|u_{x_0,2R}(t_0)|^2
\int\limits_{B(x_0,2R)}\varphi^2_{x_0,2R}dx,$$
and
$$\frac 12\int\limits_{-1}^{t_0}\int\limits_B\varphi^2_{x_0,2R}(x)
|\widehat{u}(x,t)+u_{x_0,2R}(t)|^2\pa_t\chi^2(t)dx\,dt=$$
$$\frac 12\int\limits_{-1}^{t_0}\int\limits_B\varphi^2_{x_0,2R}(x)
|\widehat{u}(x,t)|^2\pa_t\chi^2(t)dx\,dt-$$
$$-\int\limits_{-1}^{t_0}\chi^2(t)u_{x_0,2R}(t)\pa_t u_{x_0,2R}(t)dt\int\limits_{B(x_0,2R)}\varphi^2_{x_0,2R}dx+$$$$+\frac 12|u_{x_0,2R}(t_0)|^2\chi^2(t_0)
\int\limits_{B(x_0,2R)}\varphi^2_{x_0,2R}dx.$$
Now, the local energy inequality, together with the last two  identities, implies
$$\frac 12\int\limits_{B(x_0,2R)}\chi^2(t_0)|\widehat{u}(x,t_0)|^2dx+\nu \int\limits_{-1}^{t_0}\int\limits_B\chi^2(t)\varphi^2_{x_0,2R}(x)
|\na \widehat{u}(x,t)|^2dx\,dt\leq$$
$$\leq \frac 12\int\limits_{-1}^{t_0}\int\limits_B\varphi^2_{x_0,2R}(x)
|\widehat{u}(x,t)|^2\pa_t\chi^2(t)dx\,dt-\int\limits_{-1}^{t_0}\int\limits_B
\chi^2A\na \widehat{u}\cdot\na \varphi^2_{x_0,2R} \widehat{u}dx\,dt-$$
$$-\int\limits_{-1}^{t_0}\chi^2(t)u_{x_0,2R}(t)\pa_t u_{x_0,2R}(t)dt\int\limits_{B(x_0,2R)}\varphi^2_{x_0,2R}dx-$$
$$-\int\limits_{-1}^{t_0}\int\limits_B
\chi^2A\na \widehat{u}\cdot\na \varphi^2_{x_0,2R} u_{x_0,2R}dx\,dt.$$
By the (\ref{25}), the sum of the last two terms is zero and from (\ref{23}) it follows that
$$\frac 12\int\limits_{B(x_0,2R)}|\widehat{u}(x,t_0)|^2\chi^2(t_0)\varphi^2_{x_0,2R}(x)dx+$$$$+\nu\int\limits_{-1}^{t_0}\int\limits_B\chi^2(t)\varphi^2_{x_0,2R}(x)
|\na \widehat{u}(x,t)|^2dx\,dt\leq$$
$$\leq \frac 12\int\limits_{-1}^{t_0}\int\limits_B\varphi^2_{x_0,2R}(x)
|\widehat{u}(x,t)|^2\pa_t\chi^2(t)dx\,dt-$$
$$-\int\limits_{-1}^{t_0}\int\limits_B
\chi^2a\na \widehat{u}\cdot\na \varphi^2_{x_0,2R} \widehat{u}dx\,dt-$$$$-\int\limits_{-1}^{t_0}\int\limits_B
\chi^2(d-[d]_{x_0,2R})\na \widehat{u}\cdot\na \varphi^2_{x_0,2R} \widehat{u}dx\,dt.$$
Here,
$$[d]_{x_0,2R}(t)=\frac 1{|B(2R)|}\int\limits_{B(x_0,2R)}d(x,t)dx.$$
So, inequality (\ref{22}) follows if we  choose  the cut-off function $\chi$ in an appropriate way.
\end{proof}

\begin{proof}[Proof of Theorem \ref{t21}]
Using known simple arguments, we can derive from (\ref{22}) the following estimate
$$I\equiv \frac 12\int\limits_B|\widehat{u}(x,t_0)|^2
\varphi^2_{x_0,2R}(x)dx+
\int\limits_{-1}^{t_0}\int\limits_B \chi^2_{t_0,2R}\varphi^2_{x_0,2R}|\na \widehat{u}|^2dz\leq$$
$$\leq c\Big(\frac 1{R^2}\int\limits_{Q(z_0,2R)}|\widehat{u}|^2dz+\frac 1R
\int\limits_{Q(z_0,2R)}(|\na \widehat{u}|\varphi_{x_0,2R}\chi_{t_0,2R})| \widehat{u}||d-[d]_{x_0,2R}|dz\Big).$$

We now fix an arbitrary number $s\in ]1,2[$. Let us denote as usual $s'=s/(s-1)$.
Then the right hand side of the latter inequality can be estimated with the help of H\"older inequality by
$$\frac c{R^2}\int\limits_{Q(z_0,2R)}|\widehat{u}|^2dz+\frac cR
\int\limits^{t_0}_{t_0-(2R)^2}\Big(\int\limits_{B(x_0,2R)}|d-[d]_{x_0,2R}|^{s'}
dx\Big)^{\frac 1{s'}}\times$$
$$\times \Big(\int\limits_{B(x_0,2R)}(|\na \widehat{u}|\varphi_{x_0,2R}\chi_{t_0,2R})^s|\widehat{u}|^sdx\Big)^\frac 1s.$$
Applying H\"older's inequality one more time, we find
$$I\leq \frac c{R^2}\int\limits_{Q(z_0,2R)}|\widehat{u}|^2dz+$$$$+\frac cR R^{\frac n{s'}}\ess\sup\limits_{t_0-(2R)^2<t<t_0} \sup\limits_{B(x_0,2R)\subset B }\Big(\frac 1{|B(2R)|}\int\limits_{B(x_0,2R)}|d-[d]_{x_0,2R}|^{s'}
dx\Big)^{\frac 1{s'}}\times$$$$\times \int\limits^{t_0}_{t_0-(2R)^2}\Big(\int\limits_{B(x_0,2R)}|\na \widehat{u}|^2\varphi^2_{x_0,2R}\chi^2_{t_0,2R}dx\Big)^\frac 12\Big(\int\limits_{B(x_0,2R)}|\widehat{u}|^\frac {2s}{2-s}dx\Big)^\frac{2-s}{2s}\leq$$
$$\leq \frac c{R^2}\int\limits_{Q(z_0,2R)}|\widehat{u}|^2dz+\frac {c(s)}R R^{\frac n{s'}}\|d\|_{L_\infty(BMO)}\Big(\int\limits_{Q(z_0,2R)}|\na \widehat{u}|^2\varphi^2_{x_0,2R}\chi^2_{t_0,2R}dz\Big)^\frac 12\times$$
$$\times \Big(\int\limits^{t_0}_{t_0-(2R)^2}\Big(\int\limits_{B(x_0,2R)}|\widehat{u}|^\frac {2s}{2-s}dx\Big)^\frac{2-s}{s}dt\Big)^\frac 12.$$
Summarizing our efforts, we have
$$\frac 12\int\limits_B|\widehat{u}(x,t_0|^2\varphi^2_{x_0,2R}(x)dx+
\int\limits_{-1}^{t_0}\int\limits_B \chi^2_{t_0,2R}\varphi^2_{x_0,2R}|\na \widehat{u}|^2dz\leq$$
\be\la{26}\leq c(s)(1+\Gamma^2)R^{(\frac n{s'}-1)2}\int\limits^{t_0}_{t_0-(2R)^2}\Big(\int\limits_{B(x_0,2R)}|\widehat{u}|^\frac {2s}{2-s}dx\Big)^\frac{2-s}{s}dt,\ee where $\Gamma=\|d\|_{L_\infty(BMO)}$.
Now, let us discuss simple consequences of (\ref{26}) following \ci{GS}. By Poincare-Sobolev inequality, we have
for
\be\la{27}s\leq \frac n{n-1}\ee
the following inequality
$$\Big(\int\limits_{B(x_0,2R)}|\widehat{u}|^\frac {2s}{2-s}dx\Big)^\frac{2-s}{s}\leq c(s)R^{n\frac{2-s}s+2-n}\int\limits_{B(x_0,2R)}|\na {u}|^2dx.$$
Combining (\ref{27}) and  (\ref{26}), we find
$$\int\limits_B|\widehat{u}(x,t_0)|^2\varphi^2_{x_0,2R}(x)dx\leq c(s)(1+\Gamma^2)\int\limits_{Q(z_0,2R)}|\na {u}|^2dx.$$
Hence, assuming that $Q(z_0,3R)\subset Q$, we have the second estimate
\be\la{28}\ess\sup\limits_{t_0-R^2<t<t_0}\int\limits_{B(x_0,R)}|\widehat{u}|^2(x,t)dx\leq c(s)(1+\Gamma^2)\int\limits_{Q(z_0,3R)}|\na {u}|^2dx.\ee

Now, our aim is going to be the so-called reverse H\"older inequality. We first assume that the number $s$ satisfies the condition
\be\la{29}1<s<\frac {2n}{2n-1},\qquad n=2,3,...\ee
Obviously, (\ref{29}) implies (\ref{27}) and
\be\la{210}\frac {2n}{2n-1}\leq\frac 43\leq\frac {4n}{3n-2}\leq 2,\qquad n=2,3,...\ee It is not difficult to show that under assumption (\ref{29}) there exist numbers
$0<\lb<1$, $0<\mu<1$, and $1<r<2$ such that
$$\frac {2s}{2-s}=2\lb+\frac {nr}{n-r}\mu$$
$$\lb+\mu=1$$$$\frac {nr}{n-r}\mu\frac {2-s}s=1.$$
Using these exponents, we derive from (\ref{26})
$$\int\limits_{Q(z_0,R)}|\na {u}|^2dz\leq c(s)(1+\Gamma^2)R^{(\frac n{s'}-1)2}\int\limits^{t_0}_{t_0-(2R)^2}\Big(\int\limits_{B(x_0,2R)}|\widehat{u}|^{2\lb+
\frac {nr}{n-r}\mu}dx\Big)^\frac{2-s}{s}dt\leq$$
$$\leq c(s)(1+\Gamma^2)R^{(\frac n{s'}-1)2}\int\limits^{t_0}_{t_0-(2R)^2}\Big(\int\limits_{B(x_0,2R)}
|\widehat{u}|^2dx\Big)^{\frac {2-s}s\lb}\Big(\int\limits_{B(x_0,2R)}
|\widehat{u}|^\frac{rn}{n-r}dx\Big)^{\frac {2-s}s\mu}dt.$$
The last multiplier can be estimated with the help of Sobolev's inequality
$$\int\limits_{Q(z_0,R)}|\na {u}|^2dz\leq c(s)(1+\Gamma^2)R^{(\frac n{s'}-1)2}
\ess\sup\limits_{t_0-(2R)^2<t<t_0}\Big(\int\limits_{B(x_0,2R)}
|\widehat{u}(x,t)|^2dx\Big)^{\frac 12}\times$$
$$\times R^\frac {2(r-1)}r\Big(\int\limits_{Q(z_0,2R)}|\na {u}|^rdz\Big)^\frac 1r.$$
To estimate the first multiplier on the right hand side of the last inequality, one can apply (\ref{28}) in the following way
$$\int\limits_{B(x_0,2R)}|u(x,t)-u_{x_0,2R}(t)|^2dx\leq c\int\limits_{B(x_0,2R)}|u(x,t)-u_{x_0,4R}(t)|^2dx$$
$$\leq c(s)(1+\Gamma^2)\int\limits_{Q(z_0,6R)}|\na {u}|^2\,dz$$
for a.a. $t\in ]t_0-(2R)^2,t_0[$. Combining the latter inequality, we arrive at the reverse H\"older inequality
$$\frac 1{|Q(R)|}\int\limits_{Q(z_0,R)}|\na {u}|^2dz\leq
c(s)(1+\Gamma^2)^2\Big(\frac 1{|Q(6R)|}\int\limits_{Q(z_0,6R)}|\na {u}|^2dz\Big)^\frac 12\times$$$$\times\Big(\frac 1{|Q(2R)|}\int\limits_{Q(z_0,2R)}|\na {u}|^rdz\Big)^\frac 1r$$
which holds for some $r\in ]1,2[$ and for any $Q(z_0,6R)\subset Q$. This leads to a higher integrability,
see \ci{GS}.
\end{proof}

\subsection{Moser iteration}
To avoid some technical difficulties, we will assume that matrices $a$ and $b$ and solution $u$ are sufficiently smooth in $Q$. Later  we shall show how to remove this assumption.  We also assume that our function $u$ is strictly positive in the following sense
\be\la{mi1}u(z)\geq \al_R>0\qquad \forall z\in \overline{Q}(R)\ee
for any $0<R<1$.  Sometimes assumption (\ref{mi1}) is not necessary, but for simplicity we will assume
it is satisfied.
We fix the following notation
$$\varepsilon^2(m)=\Big|\frac 1{2m}-1\Big|,\qquad p=\frac{2(n+2)}n$$ and, assuming that condition (\ref{29}) holds,  let
$$q=\frac {2s}{2-s},\qquad \gamma=\frac pq>1.$$
\begin{lemma}\label{mil1} For any $m_1\geq m_0>1/2 $  and  for any $0<\varrho<r$ with $Q(z_0,r)\subset Q$, we have
\begin{equation}
\label{mi2} \sup\limits_{z\in Q(z_0,\varrho)}u^{m_1}(z)\leq\frac {c_1(n,\nu, s, \Gamma, \varepsilon_0)}{(r-\varrho)^{\frac{n+2)}q}}
\Big(\int\limits_{Q(z_0,r)}u^{m_1q}(z)dz\Big)^\frac 1q,
\end{equation}where $\varepsilon_0=\varepsilon(m_0)$.\end{lemma}

\begin{proof}
Set $w=u^m$. For any $m\neq 0$, we can derive from (\ref{\ione})
$$\frac 1{2}\int\limits_{B(x_0,r)}\psi^2\pa_t|w|^{2}dx+\frac {2m-1}m\int\limits_{B(x_0,r)}\psi^2 a\na w\cdot \na wdx=$$
\be\la{mi3}=-\Big(\int\limits_{B(x_0,r)}a\na w\cdot w\na \psi^2dx+\int\limits_{B(x_0,r)}d\na w\cdot w\na \psi^2dx\Big)\, ,\ee
with a cut-off function $\psi$ satisfying:
$$\psi(x,t)=\varphi(x)\chi(t),$$
$$\varphi(x)=1 \quad x\in B(x_0,\varrho),\qquad \varphi(x)=0 \quad x\notin B(x_0,r),$$
$$0\leq \varphi \leq 1,\qquad |\na \varphi|\leq \frac c{r-\varrho},$$
$$\chi(t)=0\quad t<t_0-r^2,\qquad \chi(t)=1 \quad t>t_0-\varrho^2,$$ $$\chi(t)=\frac {t-(t_0-r^2)}{r^2-\varrho^2}\quad t_0-r^2\leq t\leq t_0-\varrho^2.$$

Next, we introduce the following sequence of exponents
\be\la{mi4}l_0=q,\qquad l_i=\ga^il_0,\qquad i=0,1,...,\ee
If we let
$$m_i=l_im_1/p,\qquad i=1,2,...,$$
then we have
\be\la{mi5} m_iq=l_{i-1}m_1,
\qquad
\varepsilon^2(m_i)
=\frac 1{2m_i}-1>\varepsilon^2_0, \qquad i=1,2,....\ee

Letting $m=m_i$ in  (\ref{mi3}) and taking into account (\ref{mi5}), we find
$$\sup\limits_{t_0-\varrho^2<t<t_0}\int\limits_{B(x_0,\varrho)}|w(x,t)|^2dx
+\varepsilon_0^2\nu\int\limits_{Q(z_0,r)}\psi^2|\na w|^2dz\leq$$
\be\la{mi6}\leq \frac c{(r-\varrho)^2}\int\limits_{Q(z_0,r)}\psi^2| w|^2dz+c\nu^{-1}\int\limits_{Q(z_0,r)}\psi|\na\psi|w|\na w|dz\ee
$$+c\int\limits_{Q(z_0,r)}|d-[d]_{x_0,r}|\psi|\na\psi|w |\na w|dz.$$ The same arguments as in Section 2 show that the latter inequality gives us:
$$|w|^2_{2,Q(z_0,\varrho)}\equiv\sup\limits_{t_0-\varrho^2<t<t_0}
\int\limits_{B(x_0,\rho)}|w(x,t)|^2dx
+\int\limits_{Q(z_0,\varrho)}|\na w|^2dz\leq$$
$$\leq \frac {c(s,\varepsilon_0)}{(r-\varrho)^2}(1+\Gamma^2)r^{\frac {2(n+2)}{s'}}\Big(\int\limits_{Q(x_0,r)}|{w}|^qdz\Big)^
{\frac 2q }$$
with $s$ satisfying condition (\ref{29}). By the known embedding theorem, see \ci{LSU}, we have $\|w\|_{p,Q(z_0,\varrho)}\leq c |w|_{2,Q(z_0,\varrho)}$ with $p=\frac {2(n+2)}n$ and, hence,
$$\Big(\frac 1{|Q(\varrho)|}\int\limits_{Q(z_0,\varrho)}|w|^pdz\Big)^\frac 1p\leq c(s,\varepsilon_0)(1+\Gamma)\Big(\frac r{r-\varrho}\Big)\times$$\be\la{mi7}\times\Big(\frac r\varrho\Big)^\frac n2
\Big(\frac 1{|Q(r)|}\int\limits_{Q(z_0,r)}|w|^qdz\Big)^{\frac 1q}.\ee It is worth noting that, under assumption (\ref{29}) we have $p>q$.

Our further steps are  routine. We let
$$\varrho=R_i=\frac R2+\frac R{2^{i+1}},\qquad r=R_{i-1},\qquad i=1,2,...,$$
in (\ref{mi7}) and find
$$\Big(\frac 1{|Q(R_i)|}\int\limits_{Q(z_0,R_i)}|u|^{m_1l_i}dz\Big)^\frac 1{l_i}\leq
$$$$\leq (c(s,\varepsilon_0,\Gamma)2^i)^\frac 1{\gamma^{i-1}}\Big(\frac 1{|Q(R_{i-1})|}\int\limits_{Q(z_0,R_{i-1})}|u|^{m_1l_{i-1}}dz\Big)^\frac 1{l_{i-1}}$$
 for $i=1,2,...$. After iterations, we arrive at (\ref{mi3})
 with $\varrho=R/2$ and $r=R$. General case is deduced from this particular one with help of known arguments.
\end{proof}

 To see  what happens if $0<m<1/2$, we have to introduce additional notation
 $$Q^+(z_0,R)=B(x_0,R)\times ]t_0,t_0+R^2[, \qquad Q^+(R)=Q^+(0,R),$$
 $$\widetilde{Q}(z_0,R)=B(x_0,R)\times ]t_0-R^2,t_0+R^2[, \qquad \widetilde{Q}(R)=\widetilde{Q}(0,R).$$
 \begin{lemma}\label{mil2} For any $0<m_1<1/2$  and  for any $0<\varrho<r$ provided $\widetilde{Q}(z_0,r)\subset Q$, we have
\begin{equation}
\label{mi8} \sup\limits_{z\in \widetilde{Q}(z_0,\varrho)}u^{m_1}(z)\leq\frac {c_2(n,\nu, s, \Gamma)}{(r-\varrho)^{\frac{n+2}q}}
\Big(\int\limits_{\widetilde{Q}(z_0,r)}u^{m_1q}(z)dz\Big)^\frac 1q.
\end{equation}\end{lemma}

\begin{proof}
We replace the function $\chi$ with the following one
$$\chi(t)=0\quad t>t'_0+r^2,\qquad \chi(t)=1 \quad t<t'_0+\varrho^2,$$ $$\chi(t)=\frac {-t+(t'_0+r^2)}{r^2-\varrho^2}\quad t'_0+\varrho^2\leq t\leq t'_0+r^2,\qquad t'_0=t_0-\Big(\frac 34\Big)^2r^2.$$
Then from (\ref{mi3}), we can derive (an analog of (\ref{mi6}))
$$\sup\limits_{t_0<t<t_0+\varrho^2}\int\limits_{B(x_0,\varrho)}|w(x,t)|^2dx
+\varepsilon^2(m)\nu\int\limits_{Q^+(z'_0,r)}\psi^2|\na w|^2dz\leq$$
\be\la{mi9}\leq \frac c{(r-\varrho)^2}\int\limits_{Q^+(z'_0,r)}\psi^2| w|^2dz+c\nu^{-1}\int\limits_{Q^+(z'_0,r)}\psi|\na\psi|w|\na w|dz\ee
$$+c\int\limits_{Q^+(z'_0,r)}|d-[d]_{x_0,r}|\psi|\na\psi|\omega |\na \omega|dz,\qquad z'_0=(x_0,t'_0).$$
Next, it is not so difficult to check that there exists a natural number $k$
with the following property
$$\frac 1 {\ga^2}m_1\leq m_1'=\ga^{-(k-\frac 12)}\frac 12\leq m_1.$$
And then, for this number $k$, we have
$$m'_1<m'_2<...<m'_k<\frac 12< m'_{k+1}<...,$$
where
$$m'_i=\frac {l_im'_1}p=\ga^{i-1}m_1'=\ga^{i-k-\frac 12}\frac 12,\qquad i=1,2,....$$
and numbers $l_i$ is defined by (\ref{mi4}). It is easy to check that
$$\varepsilon^2(m'_i)\geq \ga^\frac 12-1, \qquad i=1,2,...,k,$$
and then repeating derivation of (\ref{mi7}) for $w=u^{m'_i}$  with the same indices $i$, we find
$$\Big(\frac 1{|Q(\varrho)|}\int\limits_{Q(z'_0,\varrho)}|w|^pdz\Big)^\frac 1p\leq c(s,\Gamma)\Big(\frac r{r-\varrho}\Big)\times$$\be\la{mi10}\times\Big(\frac r\varrho\Big)^\frac n2
\Big(\frac 1{|Q(r)|}\int\limits_{Q(z'_0,r)}|w|^qdz\Big)^{\frac 1q}.\ee
Now, we consider (\ref{mi10}) for
$$r=r_i,\qquad \varrho=r_{i-1}, \qquad r_i=\frac r4+\frac 14\frac r{2^i}$$
and find
$$\Big(\frac 1{|Q(r_i)|}\int\limits_{Q^+(z'_0,r_i)}|u|^{m'_1l_i}dz\Big)^\frac 1{l_i}\leq
$$$$\leq (c(s,\Gamma)2^i)^\frac 1{\gamma^{i-1}}\Big(\frac 1{|Q(r_{i-1})|}\int\limits_{Q^+(z'_0,r_{i-1})}|u|^{m'_1l_{i-1}}dz\Big)^\frac 1{l_{k}}$$
 for $i=1,2,...,k$. After exactly $k$ iterations, we have
 $$\Big(\frac 1{|Q(3r/4)|}\int\limits_{Q^+(z'_0,3r/4)}|u|^{m'_{k+1}q}dz\Big)^\frac 1{l_k}=\Big(\frac 1{|Q(3r/4)|}\int\limits_{Q(z_0,3r/4)}|u|^{m'_{k+1}q}dz\Big)^\frac 1{l_k}\leq
$$$$\leq c(s,\Gamma)\Big(\frac 1{|Q(r)|}\int\limits_{Q^+(z'_0,r)}|u|^{m'_1q}dz\Big)^\frac 1{q}.$$
Since $m'_{k+1}>1/2$, we are in a position to apply Lemma \ref{mil1} letting there $m_1=m_0=m'_{k+1}$ and conclude that
$$\sup\limits_{z\in Q(z_0,r/2)}u^{m'_{k+1}}(z)\leq c(s,\Gamma)
\Big(\frac 1{r^{n+2}}\int\limits_{Q(z_0,3r/4)}u^{m'_{k+1}q}dz\Big)^\frac 1q.$$
Taking into account definition (\ref{mi4}) of $l_k$ and combining the latter inequalities, we find
$$\sup\limits_{z\in Q(z_0,r/2)}u^{m_1}(z)\leq \Big[c(s,\Gamma)\Big]^{(1+\ga^{-k})\frac {m_1}{m'_1}}\Big(\frac 1 {r^{n+2}}\int\limits_{Q^+(z'_0,r)}|u|^{m'_1q}dz\Big)^{\frac 1{q}\frac {m_1}{m'_1}}$$
and, by H\"older inequality, we have
\be\la{mi11}\sup\limits_{z\in Q(z_0,r/2)}u^{m_1}(z)\leq c(s,\Gamma)\Big(\frac 1 {r^{n+2}}\int\limits_{Q^+(z'_0,r)}|u|^{m'_1q}dz\Big)^{\frac 1{q}}.\ee
We may shift in time this estimate and show that
\be\la{mi12}\sup\limits_{z\in Q^+(z_0,r/2)}u^{m_1}(z)\leq c(s,\Gamma)\Big(\frac 1 {r^{n+2}}\int\limits_{Q^+(z''_0,r)}|u|^{m'_1q}dz\Big)^{\frac 1{q}},\ee
where $z''_0=(x_0,t''_0)$ and $t''_0=t_0-\frac {5r^2}{16}.$  From (\ref{mi11}) and (\ref{mi12}),  estimate (\ref{mi8})
 with $\varrho=r/2$ follows. General case is deduced from this particular one with help of known arguments.
\end{proof}

\begin{lemma}\label{mil3} For any $\varepsilon>0 $  and  for any $0<\varrho<r$ such that  $Q(z_0,r)\subset Q$, we have
\begin{equation}
\label{mi13} \sup\limits_{z\in Q(z_0,\varrho)}u^{-\varepsilon}(z)\leq\frac {c_3(n,\nu, s, \Gamma)}{(r-\varrho)^{\frac{n+2)}q}}
\Big(\int\limits_{Q(z_0,r)}u^{-\varepsilon  q}(z)dz\Big)^\frac 1q.
\end{equation}\end{lemma}

\begin{proof}
We let $v=u^{-\varepsilon}$ and observe that by (\ref{\ione}),
the function $v$ satisfies
$$\pa_tv -\div A\na v<0.$$
We can repeat the proof of Lemma \ref{mil1} with $m_1=1$ for $v$ instead of $u$ and then show (\ref{mi13}).
\end{proof}

\subsection{Estimates of $\ln u$}
\begin{lemma}\la {el1}Assume that $u$ i a sufficiently smooth positive solution to equation (\ref{\ione}) and $Q'(z_0,R)=B(x_0,2R)\times ]t_0-R^2,t_0+R^2[\subset Q$. There exist two constants $c_4=c_4(n,\nu)$ and $a^R$ such that
\be\la{e1}|\{z\in Q^+(z_0,R):\,\,\, -\ln u-a^R>s\}|\leq \frac {c_4R^{n+2}}s,\ee\be\la{e2}|\{z\in Q(z_0,R):\,\,\, -\ln u-a^R<-s\}|\leq \frac {c_4R^{n+2}}s.\ee\end{lemma}

\begin{proof}
To simplify notation, we shift and scale our variables
in the following way
$$u^R(y,s)=u(x_0+Ry,t_0+R^2s),\qquad A^R(y,s)=A(x_0+Ry,t_0+R^2s)$$
for $(y,s)\in Q'=B(2)\times ]-1,1[$. Since equation (\ref{\ione}) is invariant with respect to this transformation, we may reduce our  considerations to the cylinder $Q'$ and, after proving  our result for this particular case, get all the statements of the lemma with the help of inverse translation and dilatation.
Without ambiguity, in what follows, we drop upper index $R$ in the notation of functions $u^R$ and $A^R$.

So, if we let $v=\ln u$, then by (\ref{\ione})
\be\la{e3}\pa_tv-\div(A\na v)+\na v\cdot a\na v=0\ee
in $Q'$. Take and fix a smooth nonnegative cut-off function $\psi=\psi(x)$ so that
$\psi=1$ in $B$ and $\psi=0$ outside $B(2)$. Multiplying equations (\ref{e3}) by $\psi^2$ and integrating the product in $x$ over $B(2)$ and in $t$ over the interval $]t_1,t_2[$, we find
$$\int\limits_{B(2)}v\psi^2dx\Big|^{t_2}_{t_1}+\int\limits^{t_2}_{t_1}
\int\limits_{B(2)}\na \psi^2\cdot A\na vdx\,dt+\int\limits^{t_2}_{t_1}
\int\limits_{B(2)} \psi^2\na v\cdot a\na vdx\,dt=0$$
and thus
$$\int\limits_{B(2)}v\psi^2dx\Big|^{t_2}_{t_1}+\int\limits^{t_2}_{t_1}
\int\limits_{B(2)} \psi^2\na v\cdot a\na vdx\,dt\leq $$$$\leq c\int\limits^{t_2}_{t_1}
\int\limits_{B(2)}\psi|\na \psi||\na v|(|a|+|d-[d]_{0,2}|)dx\,dt.$$
After application of the Cauchy-Schwartz  inequality, we have the following estimate
\be\la{e4}\int\limits_{B(2)}v\psi^2dx\Big|^{t_2}_{t_1}+\frac \nu 2\int\limits^{t_2}_{t_1}
\int\limits_{B(2)} \psi^2|\na v|^2dx\,dt\leq c( \Gamma)(t_2-t_1).\ee

From this point we essentially repeat arguments of J. Moser in \ci{M1}, see Lemma 3 therein. We do this just for completeness. As it is pointed out in \ci{M1}, we can choose our cut-off function $\psi$
so that the following Poincar\`{e}-type inequality takes place
$$\int\limits_{B(2)}|v(x,t)-V(t)|^2\psi^2(x)dx\leq c\int\limits_{B(2)}|\na v(x,t)|^2
\psi^2(x)dx,$$ where
$$V(t)=
\int\limits_{B(2)}v(x,t)\psi^2(x)dx\Big(\int\limits_{B(2)}\psi^2(x)dx\Big)^{-1}.$$
Making use of this inequality, we can derive from (\ref{e4}) the following relation
$$V(t_2)-V(t_1)+c^{-1}_4\int\limits^{t_2}_{t_1}\int\limits_{B}
|v(x,t)-V(t)|^2dx\,dt\leq c_5(n,\nu,\Gamma)(t_2-t_1)$$
which can be reduced to the  differential form
$$\frac {dV}{dt}(t)+c^{-1}_4\int\limits_{B}
|v(x,t)-V(t)|^2dx\leq c_5.$$
One may  make this inequality homogeneous with help of the shift
$$w(x,t)=v(x,t)-V(0)-c_5t,\qquad W(t)=V(t)-V(0)-c_5t.$$
This give us the inequality
\be\la{e5}\frac {dW}{dt}(t)+c^{-1}_4\int\limits_{B}
|w(x,t)-W(t)|^2dx\leq 0\ee
and the initial condition
\be\la{e6}W(0)=0.\ee

For $0<t<1$ and $s>0$, we introduce the family of sets
$$B^+_s(t)=\{x\in B:\,\,\, w(x,t)>s\,\}.$$
As it follows from (\ref{e5}) and (\ref{e6}),  for those values of parameters
$t$ and $s$, we have $w(\cdot,t)-W(t)\geq s-W(t)>0$ on $B^+_s(t)$ and, hence,
$$\frac{dW}{dt}(t)+c^{-1}_4|B^+_s(t)|(s-W)^2\leq 0$$ or
$$c_4(s-W)^{-2}\frac {d(s-W)}{dt}\geq |B^+_s(t)|$$
The latter identity can be integrated and, as a result,
we find
$$\int\limits^1_0|B^+_s(t)|dt=\{(x,t)\in Q^+:\,\,\,w(x,t)>s\,\}\leq \frac {c_4}s$$
which implies the first estimate (\ref{e1}). Other estimate (\ref{e2}) can be established in the same way.
\end{proof}

\subsection{Harnack inequality}

\begin{theorem}\la{ht1}Let $u$ be a positive sufficiently smooth solution to (\ref{\ione}) in $Q$. Then for any $Q(z_0,R)\subset Q$ we have the inequality
\be\la{h1}\sup_{z\in Q(z_{0R},R/2)}u(z)\leq c_6(n,\nu,s,\Gamma)\inf_{z\in Q(z_0,R/2)}u(z),\ee where $z_{0R}=z_0-(0,R^2/2)$.\end{theorem}

\begin{proof}
Using translation and dilatation  similar to those described in  Section 4, one may consider  our problem  in a canonical domain, say, in $Q'$. Now our aim is to make use of  estimates proved in Sections 3 and 4 plus some iteration technique in order to find a particular version of the Harnack inequality. It can  be extended to the general case of Theorem \ref{ht1} with the help of covering methods.
So, in this section, we follow \ci{M2} with minor changes.

By Lemma \ref{el1}, see (\ref{e2}), we know that
\be \la{h2}|\{z\in Q:\,\,\, -\ln u-a<-s\}|\leq \frac {c_4}s\ee
for some constant $a$. As in \ci{M2}, we introduce the following function
$$\varphi(r)=\sup\limits_{z\in \widetilde{Q}(z_0,r)} \ln w(z)$$
with $z_0=(0,1/2)$, $r\in ]1/2, 1/\sqrt{2}[$, and $w=e^au$.

Now, our aim is to show  that there exist a  constant $\delta>2$ depending only on $n$, $\nu$, $s$, and $\vartheta\in [1/2, 1/\sqrt{2}[$ such that
\be\la{h3} \varphi (\vartheta) \leq \delta. \ee
To this end, we  derive from (\ref{h2})
the  estimate
\be\la{h4}\int\limits_{\widetilde{Q}(z_0,r)}w^pdz \leq  e^{p \varphi(r)}\frac {2c_4}{\varphi(r)}+e^{\frac p2 \varphi(r)}\ee
being  true for any positive $m_1=p/q$ and for any  $r\geq \vartheta$. Let us choose $m_1$ so that terms on the right-hand side of (\ref{h4}) contribute to the sum equally.  This suggests the following value for $m_1$
\be\la{h5}m_1=\frac 1{q\varphi(r)}\ln \frac{\varphi(r)}{2c_4}.\ee
Obviously, there exists a constant $\delta_1>2$ depending only on $s$ and $c_4$
such that if
\be\la{h6}\varphi (\vartheta) >\delta_1,\ee
then $m_1$ defined above belongs to the interval $]0,1/2[$. If not, then $\delta=\delta_1$. If (\ref{h6}) holds,   estimate  (\ref{mi8}) of Lemma \ref{mil2} give us the relation
$$m_1 \varphi(\varrho)\leq \ln\frac {c_2}{(r-\varrho)^\frac {n+2}q}+\frac 1q
\ln \Big(\int\limits_{\widetilde{Q}(z_0,r)}w^pdz\Big)$$ for any $\vartheta\leq \varrho<r$.
Recalling the choice of  $m_1$, we find   from (\ref{h4}) that
$$\int\limits_{\widetilde{Q}(z_0,r)}w^pdz \leq 2e^{\frac p2 \varphi(r)}$$
and thus
$$\varphi(\varrho)\leq \frac 1{m_1}\ln\frac {c_2}{(r-\varrho)^\frac {n+2}q}+\frac 12 \varphi(r).$$
The latter inequality can be rewritten with the help of (\ref{h5}) in the following way
$$\varphi(\varrho)\leq \frac 12 \varphi(r)\Big[\frac {\ln ( {c_2}{(r-\varrho)^{-\frac {n+2}q}})}{m_1\varphi(r)}+1\Big]=\frac 12 \varphi(r)\Big[\frac {\ln c^q_2(r-\varrho)^{- (n+2)}}
{\ln \frac {\varphi(r)}{2c_4}}+1\Big].$$
Then one can consider two cases. In the first case,
$$\frac {\ln c^q_2(r-\varrho)^{- (n+2)}}
{\ln \frac {\varphi(r)}{2c_4}}\leq \frac 12$$
and thus
$$\varphi(\varrho)\leq \frac 34 \varphi(r).$$
In the opposite case,
we have
$$\varphi(\varrho)\leq\varphi(r)\leq \frac {\mu_1(n,\nu,s,\Gamma)}{(r-\varrho)^{2(n+2)}}.$$
Combining both cases,
 we find the following basic inequality
 $$\varphi(\varrho)\leq \frac 34 \varphi(r)+\frac {\mu_1}{(r-\varrho)^{2(n+2)}}$$
 for any $1/2\leq \vartheta\leq\varrho<r\leq 1/\sqrt{2}$.
It can be iterated in the known way, see \ci{M2}, and the result
of these iterations can be expresses in the form
$$\varphi(\vartheta)\leq \delta_2(n,\nu,s,\Gamma).$$ So, (\ref{h3}) is proved
with $\delta=\max\{\delta_1,\delta_2\}$.

Next, let $z_*=(0,1)$ and $v=u^{-1}$. Then as it follows from Lemmata \ref{mil3} and \ref{el1}, see (\ref{mi13}) and (\ref{e1}),
$$\sup\limits_{z\in Q(z_0,\varrho)}v(z)\leq\frac {c_3}{(r-\varrho)^{\frac{n+2)}q}}
\Big(\int\limits_{Q(z_0,r)}v^{ q}(z)dz\Big)^\frac 1q$$
and
$$|\{z\in Q(z_*,1):\,\,\,\ln(e^{-a}v)>s\}|\leq \frac {c_4}s.$$
The same arguments as above show that there exists a constant $\delta_3(n,\nu,s,\Gamma)$ such that
$$\sup_{z\in Q(z_*,\vartheta)}\ln (e^{-a}u^{-1}(z))\leq \delta_3.$$
This estimate, together with (\ref{h3}), implies a particular version of the Harnack inequality
$$\sup\limits_{z\in\widetilde{Q}(z_0,\vartheta)}u(z)\leq e^{\delta\delta_3}\inf
\limits_{z\in Q(z_*,\vartheta)}u(z).$$ The general case can be obtained from the particular case with the help of covering technique, see \ci{M2}, Lemma 4, and translation and dilatation.
\end{proof}

\subsection{Nonsmooth case}


\begin{proof}[Proof of Theorem \ref{nt1}.]
Without loss of generality, we may assume that
\be\la{n2}\na u \in L_{p_0}(Q)\ee
for some $p_0>2$. To provide (\ref{n2}), we can apply Theorem  \ref{t21} and scaling.
Obviously, one can construct smooth approximations of matrices $a$ and $d$ with the following properties:
$$a^{(j)}\to a\qquad \mbox{in}\quad L_p(Q)$$
$$d^{(j)}\to d\qquad \mbox{in}\quad L_p(Q)$$
for any $p>1$ and
$$a^{(j)}\to a\qquad \mbox{a.e. in }\quad Q$$
$$d^{(j)}\to d\qquad \mbox{a.e. in }\quad Q.$$
Moreover, matrices $A^{(j)}=a^{(j)}+d^{(j)}$,  $a^{(j)}$ and $d^{(j)}$ satisfies
conditions  (\ref{i2})--(\ref{i4}) with the same constants  and  $ \|d^{(j)}\|_{L_\infty(BMO)}\leq\|d\|_{L_\infty(BMO)}$.

Then we consider the following initial boundary value problem
$$\pa_tw^{(j)}-\div (A^{(j)}\na w^{(j)})=f^{(j)},$$
$$w^{(j)}|_{\pa'Q}=0,$$
where $\pa'Q$ is a parabolic boundary of $Q$ and
$$f^{(j)}\equiv\pa_tu-\div (A^{(j)}\na u).$$
Our claim is
$$f^{(j)}\to 0\qquad \mbox{in}\quad L_2(-1,0;H^{-1}),$$
where $H^1$ is the completion of smooth compactly supported in $B$ functions
with respect to the norm $\|u\|_{2,B}+\|\na u\|_{2,B}$. Indeed, it is not difficult
to show that
$$\|f^{({j})}\|_{L_2(-1,0;H^{-1})}\leq \Big(\int\limits_Q(|a-a^{(j)}|^2
+|d-d^{(j)}|^2)|\na u|^2dz\Big)^\frac 12.$$
So, by (\ref{n2}), the right hand of the latter inequality goes to zero.
On the other hand, for $w^{(j)}$, we have global energy estimate
$$|w^{(j)}|_{2,Q}\leq c \|f^{({j})}\|_{L_2(-1,0;H^{-1})}$$
which, in turn,  means that
\be\la{n3}|w^{(j)}|_{2,Q}\to 0.\ee

Now, we let $v^{(j)}=u-w^{(j)}$. Obviously, $v^{(j)}$ is a unique solution
to the following initial boundary value problem
$$\pa_tv^{(j)}-\div (A^{(j)}\na v^{(j)})=0,$$
$$(v^{(j)}-u)|_{\pa'Q}=0.$$
We know that $v^{(j)}$ possesses the following global properties
$$\pa_tv^{(j)}\in L_2(-1,0;H^{-1}), \qquad \na v^{(j)}\in L_2(Q)$$
and, moreover, it is nonnegative on the parabolic boundary of $Q$ and smooth inside $Q$ where the equation for $v^{(j)}$ can be reduced to the form
$$\pa_t v^{(j)}- (a^{(j)}_{kl} v^{(j)}_{,l})_{,k} - d^{(j)}_{kl,k}v^{(j)}_{,l}=0.$$
Here, comma in lower indices stands for the differentiation with respect to the corresponding spatial variable and summation over repeated indices running from 1 to $n$ is adopted.
As it was shown
in \ci{LSU}, see Chapter 3, Theorem 7.2, therein, for functions satisfying equation above, the maximum principle holds and thus $v^{(j)}$ remains to be nonnegative everywhere inside $Q$. Obviously function
$v^{(j)}+\frac 1j$ satisfies all the conditions of Theorem \ref{ht1} and, hence,
$$\sup\limits_{z\in Q(z_{0R},R/2)}v^{(j)}(z)\leq c_6\inf\limits_{Q(z_0,R/2)}v^{(j)}(z).$$
Passing to the limit as $j\to \infty$ and taking into account (\ref{n3}), we arrive at (\ref{n1}).
\end{proof}

\subsection{Liouville Theorem}

In this subsection, we assume that $u$ is an ancient suitable weak solution
to equation (\ref{\ione}) which means that it is defined on $Q_-\equiv\mathbb R^n\times ]-\infty, 0[$ and is a suitable weak solution in all parabolic balls $Q(z_0,1)$
with $z_0 =(x_0,t_0)$ for any $x_0\in\mathbb R^n$ and for any $t_0\leq 0$. Since our equation is invariant with respect to translation and usual parabolic dilatation,
such a solutions will be suitable in all parabolic balls of the form
$Q(a)$ for any positive $a$. Now, we shall show its H\"older continuity provided  it is bounded.
\begin{lemma}\la{ll1} Let $u$ be an ancient suitable weak solution
to equation (\ref{\ione}). Then there are two constants $c_7$ and $\al$ which depend only on $n$, $\nu$, $s$, satisfying condition (\ref{29}), and $\Gamma=\|d\|_{L_\infty(BMO)}$ such that
\be\la{l1+}|u(z)-u(z_0)|\leq c_7|z-z_0|_{par}^\al\sup\limits_{z\in Q_-}|u(z)|\ee
for any $z$ and $z_0$ from $Q_-$ with the parabolic distance $|z-z_0|_{par}=|x-x_0|+|t-t_0|^\frac 12$. \end{lemma}

\begin{proof}
We let
$$M_R=\sup\limits_{z\in Q(z_0,R)}u(z),\qquad M_{R/2}=\sup\limits_{z\in Q(z_0,R/2)}u(z),$$
$$m_R=\inf\limits_{z\in Q(z_0,R)}u(z),\qquad m_{R/2}=\inf\limits_{z\in Q(z_0,R/2)}u(z).$$
If we let $v(z)=M_R-u(z)$, then will be a nonnegative suitable weak solution to equation (\ref{\ione}) in parabolic ball $Q(z_0,R)$. By translation and parabolic dilatation, we can derive from Theorem \ref{ht1} the following inequality for $v$
\be\la {l2}\inf\limits_{z\in Q(z_0,R/2)}v(z)=M_R-M_{R/2}\geq \frac 1{c_6}(M_R-u(z))\ee
for all $z\in Q(z_{0R},R/2)$. On the other hand, for the same reason, we may apply Theorem \ref{ht1} to function $w=u(z)-m_R$ and find
\be\la{l3}m_{R/2}-m_R\geq \frac 1{c_6}(u(z)-m_R)\ee
for all $z\in Q(z_{0R},R/2)$.
Adding (\ref{l2}) and (\ref{l3}), we arrive at the inequality
\be\la{l4}\osc(z_0,R)-\osc(z_0,R/2)\geq \frac 1{c_6}\osc(z_0,R),\ee
where $\osc(z_0,R)=M_R-m_R$, or
$$\osc(z_0,R/2)\leq \vartheta\osc(z_0,R)$$
with $\vartheta=1-1/c_6$. After simple iterations, we have the series of the inequalities
$$\osc(z_0,R/2^k)\leq \vartheta^k\osc(z_0,R)$$
which can be reduced to the form
$$\osc(z_0,\varrho)\leq c_7\varrho^\al\vartheta\osc(z_0,R).$$
The latter is true for $z_0\in Q_-$  and  all $0<\varrho<R<+\infty$ and certainly implies (\ref{l1+}).
\end{proof}

\noindent
\begin{proof}[Proof of Theorem \ref{it2}.]
We let $M=\sup_{z\in Q_-}|u(z)|$. If we scaled  our solution $u$ and matrix $A$
so that $$u^R(y,s)=u(Ry,R^2s),\qquad A^R(y,s)=A(Ry,R^2y),$$ then as it is easy to see $u^R$ and $A^R$ satisfy the equation (\ref{\ione}) in $Q_-$ and
$$\nu=\nu^R,\qquad M=M^R=\sup\limits_{z\in Q_-}|u^R(z)|,\qquad \Gamma=\Gamma^R=
\|d^R\|_{L_\infty(BMO)}.$$
By Lemma \ref{ll1}, we have
$$|u^R(e)-u^R(0)|\leq c_7|e|_{par}^\al M^R$$
for any $e=(y,s)\in Q_-$. Making inverse scaling in the latter inequality,
we find
$$|u(z)-u(0)|\leq c_7 |z|_{par}^\al\frac 1{R^\al}M$$
for any $z\in Q_-$ and for any $R>0$. By arbitrariness of $R$, we show that
$u$ must be a constant.
\end{proof}

\section{An elementary proof of an elliptic Liouville theorem in 2D and a counterexample}

In this section we explore the Liouville theorem for \eqref{ell1} in two dimensions. Assuming that the divergence-free vector field $b$ is in the space $(BMO)^{-1}$, we provide an elementary, short, and self contained proof showing that bounded subsolutions (and supersolutions) are constant.  Afterwards, we construct a counterexample to such a Liouville theorem for a divergence-free vector field whose stream function is bounded by $\ln |x| \ln \ln |x|$ ($\notin BMO$) for large $|x|$. This construction shows that the hypothesis $b \in (BMO)^{-1}$ is quite sharp.

If $b$ is a smooth divergence-free vector field on $\bbR^2$, then it has a {\it stream function} $H:\bbR^2\to\bbR$ as in \eqref{stream2} and we have \eqref{stream3}.
This relationship between  $b$ and $A$ allows us to introduce
the notion of a weak solution for very singular drifts.

\begin{definition}\label{definition of weak solutions} Let $b$ be a divergence-free drift from $BMO^{-1}(\mathbb R^2)$, that is,
$H\in BMO(\mathbb R^2)$. We say that a function $u\in H^1_{{\rm loc}}(\mathbb R^2)$  is a weak subsolution
to \eqref{ell1} in $\bbR^2$, that is, a weak solution to
\begin{equation}\label{equation with drift}-\Delta u+b\cdot \nabla u \leq 0\end{equation}
 if for any nonnegative test function $v\in C^\infty_0(\mathbb R^2)$ we have
\begin{equation}\label{1.1}\int\limits_{\mathbb R^2} (A\nabla u)\cdot \nabla v \ dx \leq 0.\end{equation}
Weak supersolutions are defined by reversing both inequalities.
\end{definition}

We note that, as mentioned in the introduction, the bilinear form in \eqref{1.1} extends continuously to compactly supported $v\in \dot{H}^1$  and hence in \eqref{1.1}  one can equivalently consider any nonnegative compactly supported $v\in H^1(\bbR^2)$.

\begin{theorem} \lbl{T.1.1}
Assume that  $b\in BMO^{-1}(\mathbb R^2)$ is divergence-free, and let $u$ be a weak subsolution to (\ref{ell1}) in $\bbR^2$.
If $u$ is bounded then $u$ is a constant.
\end{theorem}

{\it Remark.}
If the drift $b$ is not too irregular, for example, $b\in L_{2,{\rm loc}}(\mathbb R^2)$, then distributional solutions to (\ref{equation with drift}) can be defined for $u\in  L_{2,{\rm loc}}(\mathbb R^2)$. In this case, bounded solutions are  in $H^1_{\rm loc}(\mathbb R^2)$ and satisfy (\ref{1.1}) automatically, as we will show in the next section.
\medskip

The proof of Theorem \ref{T.1.1} is an immediate consequence of the following lemma.

\begin{lemma}
Let $u$ be a bounded  weak subsolution of
\beq \lbl{e:1}
- {\rm div}  (A \nabla u) \leq 0 \qquad \text{in } \mathbb R^2
\eeq
where $A(x)=a(x)+d(x)$ with $a$ symmetric and $d$ skew symmetric.
Assume that there are $\lambda,\Lambda>0$ such that for any $x,\xi\in\bbR^2$  we have
\begin{align}
(a(x) \xi)\cdot \xi &\geq \lambda  |\xi|^2, \label{e:l-ellipticity}\\ 
||a||_{L_\infty} &\leq \Lambda,  \label{e:l-symmetric-part-in-Linf}\\ 
|| d||_{BMO} &\leq \Lambda.  \label{e:l-skew-part-in-Lp} 
\end{align}
Then $u$ is constant.
\end{lemma}

{\it Remark.}
Note that we actually prove a Liouville theorem for bounded subsolutions. This is only possible in two dimensions. In higher dimensions one needs $u$ to be a solution in order to show that it is constant even in the case of the Laplace equation.

\begin{proof}
Without loss of generality we can assume that $u$ is nonnegative (otherwise we can add a constant).
Let $\eta$ be the  test function
\[\eta(x) = \begin{cases}
1 & \text{if } |x| \leq 1 \\
1 - \frac{\log |x|}{\log R} & \text{if } 1 \leq |x| \leq R \\
0 & \text{if } |x| > R
\end{cases}\]

We take $v=u \eta^2$ in \eqref{1.1} to obtain
\begin{equation} \label{e:c1}
\begin{aligned}
0 &\geq \int (A \nabla u) \cdot \nabla(u \eta^2) \dd x \\
&\geq \int \eta^2(a\nabla u) \cdot \nabla u  \dd x + 2\int u\eta(A \nabla u) \cdot \nabla \eta  \dd x\\
&\geq \lambda \int |\grad u|^2 \eta^2 \dd x + 2\int u\eta(A \nabla u) \cdot \nabla \eta  \dd x
\end{aligned}
\end{equation}
(all integrals are over $\mathbb R^2$ unless otherwise indicated).
Therefore we have
\begin{equation} \label{e:d0}
\lambda \int |\grad u|^2 \eta^2 \dd x \leq 2 \left|\int u\eta(A \nabla u) \cdot \nabla \eta  \dd x \right|
\end{equation}

We need to estimate the second term.
For the symmetric part of $A$, we have
\begin{align}
\left| \int u\eta(a \nabla u) \cdot \nabla \eta  \dd x\right| &\leq \Lambda || u \grad \eta ||_{L_2} || \eta \grad u||_{L_2} \notag\\
& \leq \frac \lambda 8 || \eta \grad u||_{L_2}^2 + C ||\grad \eta ||_{L_2}^2 \notag \\
&\leq \frac \lambda 8 || \eta \grad u||_{L_2}^2 + \frac C {\log R} \label{e:d1}
\end{align}
for a constant $C$ depending only on $||u||_{L_\infty}$, $\Lambda$, and $\lambda$.

Let $\bar k$ be the average of $k$ in $B(R)$, the disk of radius $R$ centered at the origin.  It is easy to check that
\begin{align*}
\int u\eta (\bar  d \nabla  u)\cdot \nabla   \eta  \dd x &=0
\end{align*}

Now we estimate the contribution to the variable skew-symmetric part of the coefficients using H\"older inequality:
\begin{align}
\left| \int u\eta(d \nabla u) \cdot \nabla \eta  \dd x\right|  &\leq
C \left( \int\limits_{B(R)} |d -\bar d|^4 \right)^{1/4} \left( \int u^4 |\grad \eta|^4 \right)^{1/4} \left(\int \eta^2 |\grad u|^2 \right)^{1/2} \notag \\
&\leq \frac 14 \int \eta^2 |\grad u|^2 \dd x + C \left( \int\limits_{B(R)} |d - \bar d|^4 \right)^{1/2} \left( \int u^4 |\grad \eta|^4 \right)^{1/2} \label{e:c2}
\end{align}
Since $d$ is a BMO function, we have
\[ \int\limits_{B(R)} |d - \bar d|^4 \dd x\leq C R^2 \]
and by direct computation using that $u$ is bounded,
\[ \int u^4 |\grad \eta|^4 \dd x\leq \frac C {R^2 (\log R)^4} \]
Therefore, \eqref{e:c2} gives
\begin{equation} \label{e:d2}
\left| \int u\eta (d \nabla u)\cdot \nabla  \eta   \dd x \right| \leq \frac 14 \int \eta^2 |\grad u|^2 \dd x+ \frac C {(\log R)^2}
\end{equation}
Adding \eqref{e:d1} and \eqref{e:d2}, we estimate the right hand side of \eqref{e:d0} to obtain
\[ \int\limits_{B(1)} |\grad u|^2 \dd x \leq \int\limits_{\mathbb R^2} |\grad u|^2 \eta^2 \dd x \leq \frac{C}{\log R} + \frac C {(\log R)^2} \]
for a constant $C$ independent of $R$. We conclude the proof by taking $R \to \infty$.
\end{proof}

\begin{proof}[Proof of Theorem \ref{T.1.2}]
Let $h:\bbR^+_0\to\bbR^+_0$ be such that $h(s)=e^{1-e}s$ for $s\in[0,e^e]$ and $h(s)= \ln s\ln\ln s$ for $s\ge e^e$. For $x=(x_1,x_2)$ define $\hat x \equiv \min\{|x_1|,|x_2|\}$ and
\beq \lbl{1.2}
\til H(x) \equiv C \sgn(x_1x_2) h(\hat x)
\eeq
with $C$ large. If now $H\equiv \eta*\til H$ for some radially symmetric smooth mollifier $\eta$ supported on the unit disc and $b\equiv \nabla^\perp H$, then the hypotheses of the theorem are satisfied. Moreover,  if $K^\pm\equiv\{x\,\big|\, \pm x_2\ge |x_1|+2 \}$, then for some $c>0$ independent of $C$ we have
\beq \lbl{1.3}
b_1(x)  = 0 \qquad \text{and} \qquad \sgn(x_2) b_2(x)  \ge cCh' (\hat x) \qquad \text{if $x\in K^+\cup K^-$.}
\eeq

Let $B_t$ be the 2-dimensional Brownian motion with $B_0=0$, defined on some probability space $(\Omega, \calF,\bbP)$, and for $x\in\bbR^2$ let $X_t^x=X_t^x(\omega)$ be the stochastic process with $X_0^x=x$ and satisfying the SDE
\[
dX_t^x = b(X_t^x)dt + \sqrt 2 dB_t.
\]
It is then well known (see, e.g., \cite{Oks}) that if $v$ solves
\[
\partial_tv + b\cdot\nabla v - \Delta v = 0
\]
on $\bbR^2$ with $v(0,x)=v_0(x)$, then
\beq \lbl{1.4}
v(t,x) = \bbE(v_0(X_t^x)).
\eeq

We let $v_0=2\chi_{\{x\,|\, x_2>0\}}$. Clearly $v(t,x)\in[0,2]$ by \eqref{1.4}, and $(b_1(x_1,-x_2),b_2(x_1,-x_2))=(b_1(x_1,x_2),-b_2(x_1,x_2))$ gives
\beq \lbl{1.5}
v(t,x_1,0)=1 \qquad \text{for all $(t,x_1)$.}
\eeq
So $\partial_tv(t,x_1,0)=0$ and obviously $\sgn(x_2)\partial_tv(0,x_1,x_2)\le 0$ for $x_2\neq 0$ because of $v(t,x)\in[0,2]$ and the choice of $v_0$. This and the maximum principle for $\partial _tv$ give $\sgn(\partial_tv(t,x))= -\sgn(x_2)$ and so there exists $u(x)\equiv \lim_{t\to\infty} v(t,x)$. Parabolic regularity shows that $u$ is a (bounded) solution of \eqref{ell1}.

Let $A_{t,x}\equiv \{\omega :\, (X_s^x(\omega))_2\neq 0 \text{ for all $s\in[0,t]$}  \}$ and $A_x=\bigcup_{t>0} A_{t,x}$. Then \eqref{1.4}, \eqref{1.5}, and the strong Markov property for $X_t^x$ imply
\begin{align*}
v(t,x) & = (1+\sgn(x_2))\bbP( A_{t,x} ) + 1-\bbP( A_{t,x} ),
\\  u(x) & =(1+\sgn(x_2))\bbP( A_{x} ) + 1-\bbP( A_{x} ).
\end{align*}
We will now show that $\bbP( A_{(0,4)} ) = \bbP( A_{(0,-4)} ) >0$ (the equality holds by symmetry), which implies $u(0,4)>1>u(0,-4)$.

The law of iterated logarithm (see, e.g., \cite{Oks}) implies that $\bbP(A)>0$ for
\[
A\equiv\{\omega : \, |B_t|<(3t\ln\ln t)^{1/2} \text{ for all $t\ge e^e$ and $|B_t|< 1$ for all $t\in[0,e^e]$} \}.
\]
For any $\omega\in A$, let $t>0$ be the first time such that $X_t^{(0,4)}(\omega) \in\partial K^+$. Then clearly $t\ge e^e$ and so
\beq \lbl{1.6}
|(X_s^{(0,4)})_1| = |\sqrt 2 (B_s)_1| \le (6t\ln\ln t)^{1/2}
\eeq
for all $s\in[0,t]$. Since $X_s^{(0,4)}\in K^+$ for these $s$, we have $\hat X_s^{(0,4)}=|(X_s^{(0,4)})_1|$ and \eqref{1.3} gives
\[
b_1(X_s^{(0,4)})=0 \qquad \text{and} \qquad b_2(X_s^{(0,4)})\ge cC \frac{\ln\ln (6t\ln\ln t)^{1/2} +1} {(6t\ln\ln t)^{1/2}} \ge cC\left( \frac{\ln\ln t}{6t}\right)^{1/2}
\]
for  $s\in[0,t]$. This means
\[
(X_t^{(0,4)})_2 \ge  cC \left( \frac{\ln\ln t}{6t} \right)^{1/2}t + \sqrt 2 (B_t)_2 +4 \ge \left(\frac{cC}{\sqrt 6} - \sqrt 6 \right)  (t\ln\ln t)^{1/2}.
\]
If we choose $C\ge 18c^{-1}$, then this and \eqref{1.6} give
\[
(X_t^{(0,4)})_2 \ge 2(6t\ln\ln t)^{1/2} \ge (6t\ln\ln t)^{1/2} + (6e^e)^{1/2} \ge |(X_s^{(0,4)})_1| + 9,
\]
contradicting $X_t^{(0,4)}\in\partial K^+$. Therefore $X_t^{(0,4)}(\omega) \in K^+$ for all $t\ge 0$ and $\omega \in A$. This means that $A\subseteq A_{(0,4)}$ and so $0<\bbP(A_{(0,4)})=\bbP(A_{(0,-4)})$. Hence $u(0,4)>1>u(0,-4)$ and the result follows.
\end{proof}


\section{
On a modulus of continuity in 2D and a counterexample in 3D}

In this section we prove that distributional solutions of \eqref{ell1}
with a divergence-free $b$ in a two dimensional domain $\Omega \subset \mathbb R^2$ are continuous with a logarithmic modulus of continuity that we estimate explicitly. The modulus of continuity depends on a local bound for the $H^1$ norm of $u$ and it does not essentially depend on any quantity associated with the vector field $b$. If $b \in L_{1,loc}$, then for suitable $u$ (see Theorem \ref{limit}) we can estimate the modulus of continuity in terms of the $L_\infty$ norm of $u$ instead of $H^1$ thanks to a local energy inequality. Because of the low regularity assumed for the vector field $b$, the a priori estimates are hard to extend to distributional solutions and this presents some technical difficulties that are explained below.
The estimate is a version of the classical result that functions in the border-line Sobolev spaces
which satisfy the maximum principle\footnote{Sometimes the terminology
 ``monotone in the sense of Lebesgue" is used in this context,  see e.\ g.\ \cite{Morrey}.} are continuous, with logarithmic modulus
of continuity.

We also show in this section that the same type of regularity result does not hold in three dimensions. Indeed, we construct an example of a function $u\in L_\infty(B)\cap H^1(B)$ (recall that $B=B(0,1)$ is the unit ball) and a vector field $b \in L_1(B)$, such that $u$ solves \eqref{ell1} in the distributional sense and is discontinuous at the origin.



\begin{proposition} \label{p:h1-from-linfty} Assume that the drift $b\in L_{2,{\rm loc}}(\Omega)$ is divergence-free
and
$u\in L_\infty(\Omega)$ solves \eqref{ell1} in the distributional sense (that is,
\begin{equation}\label{weak-sol}
\int\limits_\Omega b\cdot \nabla v dx = 0= \int\limits_\Omega(u\Delta v+bu\cdot \nabla v)dx
\end{equation}
for all $v\in  C^\infty_0(\Omega)$).
Then $u\in H^1_{\rm loc}(\Omega)$
and if $B(x_0,r) \Subset \Omega$, there is a constant $C$ depending only on $r$
such that
\[||\grad u||_{L_2(B{(x_0,r/2))} }\leq C \left(1+||b||_{L_1(B(x_0,r))}\right)^{1/2} ||u||_{L_\infty(B(x_0,r))}\]
\end{proposition}

\begin{proof}
The claim $u\in H^1_{\rm loc}(\Omega)$ is obvious from $bu\in L_{2,{\rm loc}}(\Omega)$.

This means that
\beq\lbl{distr}
\int\limits_\Omega b\cdot \nabla v dx =0= \int\limits_\Omega(\nabla u-bu)\cdot \nabla vdx
\eeq
for  $v\in H^1 (\Omega)$ compactly supported in $\Omega$.
Let $\eta$ be a smooth bump function such that
\begin{align*}
\eta &= 1 \qquad \text{in } B(x_0,r/2) \\
\eta &= 0 \qquad \text{in } \Omega \setminus B(x_0,r).
\end{align*}
We  take $v=u \eta^2$ in \eqref{distr} to obtain
\begin{align*}
0 &= \int\limits_{B(x_0,r)} |\grad u|^2 \eta^2 + 2 u \eta \grad u \cdot \grad \eta- b \cdot \grad   u u \eta^2 - 2 b \cdot \grad  \eta u^2\eta \dx \\
&= \int\limits_{B(x_0,r)} |\grad u|^2 \eta^2 + 2 u \eta \grad u \cdot \grad \eta - b \cdot \grad \eta u^2 \eta \dx,
\end{align*}
where we have used
\[
0= \int\limits_{B(x_0,r)} b\cdot \nabla(u^2\eta^2)dx =  2 \int\limits_{B(x_0,r)}  b \cdot \grad \eta u^2 \eta + b\cdot\nabla u u \eta^2 dx
\]
with $u^2\eta^2\in H^1 (\Omega)$. Therefore
\begin{align*}
\int\limits_{B(x_0,r)} |\grad u|^2 \eta^2 \dx &= \int\limits_{B(x_0,r)} -2 u \eta \grad u \cdot \grad \eta + b \cdot \grad \eta \, u^2 \eta \dx \\
&\leq \frac 12 \int\limits_{B(x_0,r)} |\grad u|^2 \eta^2 \dx + \int\limits_{B(x_0,r)} 2 u^2 |\grad \eta|^2 + b \cdot \grad \eta \, u^2 \eta \dx
\intertext{and thus}
\frac 12 \int\limits_{B(x_0,r/2)} |\grad u|^2 \dx &\leq \int\limits_{B(x_0,r)} 2 u^2 |\grad \eta|^2 + b \cdot \grad \eta \, u^2 \eta \dx \\
&\leq C(1+||b||_{L_1(B(x_0,r))}) ||u||_{L_\infty}^2.
\end{align*}
\end{proof}

It is worth noting that all statements of Proposition \ref{p:h1-from-linfty} hold true in higher dimensions.

We now find the modulus of continuity for functions satisfying the maximum principle and a bound in $H^1$. Note that in two dimensions, the space $H^1$ is borderline with respect to the Sobolev embeddings to spaces of continuous functions. The monotonicity of $\osc_{\partial B(r)} u$ is the extra assumption used in the theorem below to actually obtain an explicit modulus of continuity.

\begin{theorem} \label{t:2d-continuity}
Let $u \in H^1(B)$ and assume that for any $r \in (0,1)$ the maximum principle holds in $B(r)$:
\begin{align*}
\max_{B(r)} u &= \max_{\partial B(r)} u,\\
\min_{B(r)} u &= \min_{\partial B(r)} u.
\end{align*}
Then $u$ satisfies the following modulus of continuity estimate at the origin
\[ \sup_{x \in B(r)} |u(x)-u(0)| \leq \frac C {\sqrt{-\log r}} ||\grad u||_{L_2(B)} \]
for some constant $C$ independent of $u$.
\end{theorem}

\begin{proof}
Let $r \in (0,1)$. We want to estimate $\osc_{B(r)} u = \max_{B(r)} u - \min_{B(r)} u$.
\begin{align*}
\int\limits_{B \setminus B(r)} |\grad u|^2 \dx &= \int\limits_r^1 \int\limits_{\partial B(s)} |\grad u|^2 \dd \sigma \dd s \\
\intertext{since $|\grad u|^2 = u_\sigma^2 + u_\nu^2$  where $u_\sigma$ is the tangential derivative and $u_\nu$ is the normal one,}
&\geq \int\limits_r^1 \int\limits_{\partial B(s)} |u_\sigma|^2 \dd \sigma \dd s\\
\intertext{Rewriting the integral using polar coordinates $(s \theta = \sigma)$,}
&= \int\limits_r^1 \frac 1 s \int\limits_{\partial B} |u_\theta(s\theta)|^2 \dd \theta \dd s. \\
\intertext{Since $H^1(\partial B) \subset C^\alpha(\partial B)$ from the one dimensional Sobolev imbedding,}
&\geq \int\limits_r^1 \frac C s (\osc_{\partial B(s)} u)^2 \dd s.
\intertext{From the maximum principle, $\osc_{\partial B(s)} u$ is monotone in $s$, therefore,}
&\geq \int\limits_r^1 \frac C s (\osc_{\partial B(r)} u)^2 \dd s = (-C \log r) (\osc_{\partial B(r)} u)^2
\end{align*}
Taking square roots of both sides we obtain
\[ \osc_{B(r)} u = \osc_{\partial B(r)} u \leq \frac C {\sqrt{-\log r}} ||\grad u||_{L_2(B)} \]
\end{proof}

Consider now a drift $b\in L_1(B)$ and let $u\in L_\infty(B)$ be a distributional solution to (\ref{ell1}). We are interested in whether $u$ is still a continuous function and, if so, how to estimate its modulus of continuity. We do not know the answer to this question. However, it is in the affirmative if $u$ is an appropriate limit of solutions with $L_2$ drifts, as in Theorem \ref{limit}.


\begin{proof}[Proof of Theorem \ref{limit}.]
The first claim is immediate from the definition of distributional solutions. Moreover,
Proposition \ref{p:h1-from-linfty} and Theorem \ref{t:2d-continuity} show that $u_m$ are locally uniformly bounded in $H^1$ as well as locally uniformly continuous with the  modulus of continuity from \eqref{modcont}, and the second claim follows.
\end{proof}


Finally, we show that that Theorem \ref{limit} does not hold in higher dimensions  in general.





\begin{proof}[Proof of Theorem \ref{T.3.4}.]
 As in the previous section, we will again consider  vector fields with $b(Rx)=Rb(x)$, where $R(x_1,x_2,x_3) \equiv (x_1,x_2,-x_3)$. In addition, $b$ will be axisymmetric with respect to the $x_3$-axis and with no angular component.  Such divergence-free vector fields can again be obtained from a ``stream function'' $H:\mathbb R^+_0\times\mathbb R \to\mathbb R$ with $H(0,z)=0$ as
\beq \lbl{2.1}
b_H(x) \equiv \nabla\times \left[ \frac{H(\rho,z)}{2\rho^2}(-x_2,x_1,0) \right] = \frac{1} {2\rho^2}\left(x_1H_z(\rho,z), x_2H_z(\rho,z), -\rho H_\rho(\rho,z) \right),
\eeq
where $\rho\equiv\sqrt{x_1^2+x_2^2}$ and $z\equiv x_3$.  Notice that again we have $b_H\cdot \nabla H=0=b_H\cdot(x_2,-x_1,0)$, so $H$ is constant on the streamlines of $b_H$, and $b_H$ has no angular component.

We now pick $\alpha\in(\tfrac 23,1)$ and for $\rho^2+z^2<1$ and $\rho\ge 0$ we let
\beq \lbl{2.2}
\til H(\rho,z) \equiv \sgn(z)
 \begin{cases}
\rho^2 z^{-2} & \text{$\rho^\alpha \le |z| $}, \\
z^2 \rho^{-2} & \text{$ |z|^\alpha \le \rho $}, \\
(\rho |z|)^{2(1-\alpha)/(1+\alpha)} & \text{$|z|< \rho^\alpha$ and $\rho< |z|^\alpha$}.
\end{cases}
\eeq
Finally,
for some large $C$ we define $H_0\equiv C \til H$ and $b_0\equiv b_{H_0}$.

Notice that $H_0$ is continuous and vanishes on the axes, and $b_0(Rx)=Rb_0(x)$.  We also have
\beq \lbl{2.3}
-b_0(x)=C\frac x{|x_3|^3} \qquad \text{for $(x_1^2+x_2^2)^{\alpha/2} \le |x_3|$}
\eeq
as well as $b_0\in L_1(B)$ (because $|b_0(x)|\le c|x|^{-2}$ for some $c>0$ due to $\alpha\in(\tfrac 23,1)$).
Moreover, $\til H$ is smooth except on 
\[
P\equiv \left\{ \rho^2+z^2<1 :\, \text{$\rho\ge 0$ and $|z|\in\{\rho^\alpha, \rho^{1/\alpha},0\}$ } \right\},
\]
so $b_0$ is smooth except on
\[
S\equiv \left\{ x\in B :\, \left(\sqrt{x_1^2+x_2^2}\,,x_3 \right)\in P \right\}.
\]
We therefore let $H_\eps$ be smooth such that $H_\eps(\rho,-z)= -  H_\eps(\rho,z)$ and $H_\eps=H_0$ outside the $\eps$-neighborhood of $P$, the vector field $b_\eps\equiv b_{H_\eps}$ is also smooth and $|b_\eps(x)|\le c|x|^{-2}$ on $B$, as well as
\beq \lbl{2.5}
\lim_{\eps\to 0} \|b_\eps- b_0\|_{L_1(B)}= 0.
\eeq
Clearly $\nabla\cdot b_\eps=0$ for $\eps>0$, so \eqref{2.5} gives $\nabla\cdot b_0=0$ in the distributional sense.

For each $\eps>0$ we now construct (smooth) $u_\eps$ using  $b_\eps$ in a way similar to our construction of $u$ using $b$ in the proof of Theorem \ref{T.1.2}.  We are here on $B$ so we set boundary conditions $u_\eps(x)=\sgn(x_3)$ on $\partial B$ and thus consider the stochastic process \[
dX_t^{x,\eps} = -b_\eps(X_t^{x,\eps})dt + \sqrt 2 dB_t.
\]
with $X_0^{x,\eps}=x$ and stopping time
\[
\tau_\eps\equiv \inf \big\{ t\ge 0 \,\big|\, X_t^{x,\eps}\in \partial B \big\}.
\]
We therefore obtain
\beq \lbl{2.6}
-\Delta u_\eps + b_\eps \cdot \nabla u_\eps =0
\eeq
where
\beq \lbl{2.7}
u_\eps(x)=\bbE \left( \sgn \left(  \left(X_{\tau_\eps}^{x,\eps}\right)_3 \right)  \right) = \sgn(x_3) \bbP \left(\left(X_t^{x,\eps}\right)_3 \neq 0 \text{ for all $t\in [0,\tau_\eps]$}\right),
\eeq

Each $u_\eps$ is a smooth solution of \eqref{2.6} and they are uniformly H\" older continuous away from the origin due to $|b_\eps(x)|\le c|x|^{-2}$.  Therefore there is a sequence $\eps_k\to 0$ and $u_0$  such that $u_{\eps_k}\to u_0$ locally uniformly on $B\setminus\{0\}$.  (In fact,
\[
\lim_{\eps\to 0} u_\eps(x)= u_0(x) = \sgn(x_3) \bbP \left(\left(X_t^{x,0}\right)_3 \neq 0 \text{ for all $t\in [0,\tau_0]$}\right)
\]
for any $x\in B$, provided we set $u_0(0)\equiv 0$.)  But this, $\|u_\eps\|_{L_\infty} \le 1$, and \eqref{2.5} show that
\[
\int\limits_{B} (u_0\Delta v + b_0u_0\cdot \nabla v) dx = \lim_{k\to\infty} \int\limits_{B} (u_{\eps_k}\Delta v + b_{\eps_k} u_{\eps_k}\cdot \nabla v) dx = 0
\]
for any  $v\in C_0^\infty(B)$. Thus
\beq \lbl{2.8}
-\Delta u_0 + b_0 \cdot \nabla u_0 =0
\eeq
in the distributional sense. The  proof of Proposition \ref{p:h1-from-linfty} applies to each $u_{\eps_k}$, implying that their weak limit $u_0\in H^1_{\rm loc}(B)$. Since for $\eps\ge 0$ we have $u_\eps(Rx)=-u_\eps(x)$ and $||b_\eps||_{L_1(B)}$ is uniformly bounded, we only need to show
\beq \lbl{2.9}
\lim_{z\downarrow 0} \lim_{\eps\to 0} u_\eps(0,0,z) >0
\eeq
to conclude the proof of both (i) and (ii).

In fact, let us consider instead of $(0,0,z)$ with $z>0$ any $y\in B$ with $y_3>0$ and $\sqrt{y_1^2+y_2^2} \le \tfrac 12 y_3^{1/\alpha}$. Let $K_y$ be the cut-off cone
\[
K_y\equiv \left\{ x\in \bbR^3 :\, \sqrt{x_1^2+x_2^2} \le \frac {(2y_3)^{1/\alpha}}{4y_3}x_3 \right\} \subseteq \bbR^2\times\bbR^+,
\]
with upper and lower base consisting of discs $D_y$, $E_y$ centered at $(0,0,2y_3)$, $(0,0,\tfrac 12 y_3)$ and with radii $\tfrac 12(2y_3)^{1/\alpha}$, $\tfrac 18 (2y_3)^{1/\alpha}$.  Notice that its tip would be at the origin, were it not cut off.  Then $\sqrt{x_1^2+x_2^2} \le x_3^{1/\alpha}$ on $K_y$ because $\alpha>\tfrac 23$, so \eqref{2.3} holds on $K_y\cap B$. Let $\sigma$ be the exit time of $X^{y,\eps}_t$ from $K_y\cap B$, which is the same for all $\eps\lesssim y_3^{1/\alpha}$ because $b_\eps\equiv b_0$ on $K_y\cap B$ in that case.  We will show that
\beq \lbl{2.10}
\bbP(X^{y,\eps}_\sigma \in D_y\cup \partial B) \ge 1-e^{-y_3^{-3+2/\alpha}},
\eeq
provided $C$ from the definition of $H_0$ is large.  This is sufficient since $X^{y,\eps}_\sigma \in D_y\cup \partial B$ means either
\beq \lbl{2.11}
\left(X_t^{y,\eps}\right)_3 \neq 0 \quad \text{ for all $t\in [0,\tau_\eps]$}
\eeq
or $(X^{y,\eps}_\sigma)_3=2y_3$. In the latter case we have $((X^{y,\eps}_\sigma)_1^2 + (X^{y,\eps}_\sigma)_2^2)^{1/2} \le \tfrac 12 (X^{y,\eps}_\sigma)_3^{1/\alpha}$, so we can bootstrap \eqref{2.10} and obtain \eqref{2.11} with probability at least
\[
\prod_{k=1}^{\lfloor -\log_2 y_3 \rfloor} (1-e^{-(2^ky_3)^{-3+2/\alpha}}) \ge \prod_{j=-\infty}^{0} (1-e^{-(2^{j})^{-3+2/\alpha}}) \equiv m >0.
\]
Thus \eqref{2.7} gives $u_\eps(y) \ge m$ for any  $y\in B$ such that $y_3>0$ and $\sqrt{y_1^2+y_2^2} \le \tfrac 12 y_3^{1/\alpha}$  and any $\eps\lesssim y_3^{1/\alpha}$.  The claim of \eqref{2.9} now follows immediately.

It remains to prove \eqref{2.10} for some large $C$  independent of $y$.  The point here is that $X^{y,\eps}_t$ starts well inside $K_y$ (specifically, $\dist(y,\partial K_y)\ge dy_3^{1/\alpha}$ for some $d\in(0,1)$) and the (strong) drift $-b$ quickly pushes it towards $D_y$ while the Brownian term will not affect this picture much during the short time needed to reach $D_y$, at least with probability close to 1. We have
\beq \lbl{2.12}
X^{y,\eps}_\sigma=y-\int_0^\sigma b_0(X^{y,\eps}_t)dt + B_\sigma,
\eeq
as well as
\beq \lbl{2.13}
\bbP( |B_t|< dy_3^{1/\alpha} \text{ for all $t\in[0,8C^{-1}y_3^3]$}) \ge 1-e^{-y_3^{-3+2/\alpha}},
\eeq
provided $C$ is large enough. Since the vector $-b_0(X^{y,\eps}_t)$ points `inside' the mantle of $K_y$ for $t\in[0,\sigma)$ (because of \eqref{2.3} and the fact that the cut-off tip of $K_y$ is the origin) and has a positive third component, and $\dist(y,\partial K_y)\ge dy_3^{1/\alpha}$, this means that with probability at least $1-e^{-y_3^{-3+2/\alpha}}$, the process $X^{y,\eps}_t$ cannot exit $K_y$ through the mantle or the bottom $E_y$ before time $8C^{-1}y_3^3$. But we have $\left(-b_0(X^{y,\eps}_t)\right)_3\ge C(2y_3)^{-2}$ for $t\in[0, \sigma]$, so \eqref{2.12}, \eqref{2.13}, and $y_3 +C(2y_3)^{-2} 8C^{-1}y_3^3 -cy_3^{1/\alpha}\ge 2y_3$ yield
\[
\bbP(X^{y,\eps}_\sigma \in D_y\cup \partial B \text{ and } \sigma\le 8C^{-1}y_3^3) \ge 1-e^{-y_3^{-3+2/\alpha}}.
\]
This proves \eqref{2.10} and the result follows.
\end{proof}


\end{document}